\documentclass{article}
\usepackage{amsfonts}
\usepackage{amsmath}
\usepackage{dsfont}
\usepackage{graphicx}
\usepackage{amssymb}
\usepackage{caption}
\usepackage[lined,boxed,commentsnumbered]{algorithm2e}
\usepackage{geometry}
\RequirePackage[OT1]{fontenc}
\newtheorem{theorem}{Theorem}

\newtheorem{lemma}[theorem]{Lemma}

\newtheorem{proposition}[theorem]{Proposition}
\newtheorem{remark}[theorem]{Remark}

\newenvironment{proof}[1][Proof]{\noindent\textbf{#1.} }{\ \rule{0.5em}{0.5em}}

\begin{document}

\title{Towards zero variance estimators for rare event probabilities}
\author{Michel Broniatowski and Virgile Caron}
\maketitle

\begin{abstract}
Improving Importance Sampling estimators for rare event probabilities
requires sharp approximations of conditional densities. This is achieved for
events $E_{n}:=\left( u(X_{1})+...+u(X_{n}\right) )\in A_{n}$ where the
summands are i.i.d. and $E_{n}$ is a large or moderate deviation event. The
approximation of the conditional density of the vector $\left(
X_{1},...,X_{k_{n}}\right) $ with respect to $E_{n}$ on long runs, when $%
k_{n}/n\rightarrow1$, is handled. The maximal value of $k_{n}$ compatible
with a given accuracy is discussed; simulated results are presented, which
enlight the gain of the present approach over classical IS schemes. Detailed
algorithms are proposed.
\end{abstract}

\section{Introduction and notation}

\label{sec:intro_chap2}

\subsection{Motivation and context}

Importance Sampling procedures aim at reducing the calculation time which is
necessary in order to evaluate integrals, often in large dimension. We
consider the case when the integral to be numerically computed is the
probability of an event defined by a large number of random components; this
case has received quite a lot of attention, above all when the event is of 
\textit{small} probability, typically of order $10^{-8}$ or so, as occurs
frequently in industrial applications or in communication devices.The
present paper proposes estimators for both large and moderate deviation
probabilities; this latest case is of interest for statistics. The situation
which is considered is the following.

The r.v's $\mathbf{X,X}_{i}^{\prime }s$ are i.i.d. with known common density 
$p_{\mathbf{X}}$ on $\mathbb{R}$, and $u$ is a real valued measurable
function defined on $\mathbb{R}$. Define $\mathbf{U}:=u(\mathbf{X})$ with
density $p_{\mathbf{U}}$ and 
\begin{equation*}
\mathbf{U}_{1,n}:=\sum_{i=1}^{n}\mathbf{U}_{i}.
\end{equation*}%
We intend to estimate 
\begin{equation*}
P_{n}:=P\left( \mathbf{U}_{1,n}\in nA\right) 
\end{equation*}%
for large but fixed $n$ where 
\begin{equation}
A:=(a_{n},\infty )  \label{A}
\end{equation}%
and $a_{n}$ is a convergent sequence. The limit of this sequence  either
equals $E\mathbf{U}$ or is assumed to be larger than $E\mathbf{U.}$ In the
first case it will be assumed that $a_{n}$ converges slowly in such a way
that $P\left( \mathbf{U}_{1,n}\in nA\right) $ is not obtainable through the
central limit theorem; we may call this case a moderate deviation case. The
second situation is classically referred to as a large deviation case.

The basic estimate of $P_{n}$ is defined as follows: generate $L$ i.i.d.
samples $X_{1}^{n}(l)$ with underlying density $p_{\mathbf{X}}$ and define 
\begin{equation*}
P^{(n)}(A):=\frac{1}{L}\sum_{l=1}^{L}\mathds{1}_{\mathcal{E}_{n}}\left(
X_{1}^{n}(l)\right)
\end{equation*}
where 
\begin{equation}
\mathcal{E}_{n}:=\left\{ (x_{1},...,x_{n})\in \mathbb{R}^{n}:\left(
u\left(x_{1}\right)+..+u\left(x_{n}\right)\right)\in nA\right\}  \label{E_n}
\end{equation}%
with $u_{i}:=u\left( x_{i}\right).$ The Importance Sampling estimator of $%
P_{n}$ with sampling density $g$ on $\mathbb{R}^{n}$ is

\begin{equation}  \label{FORM_IS}
P_{g}^{(n)}(A):=\frac{1}{L}\sum_{l=1}^{L} \hat{P}_{n}(l)\mathds{1}_{\mathcal{%
E}_{n}}\left( Y_{1}^{n}(l)\right)
\end{equation}
where $\hat{P}_{n}(l)$ is called "importance factor" and writes 
\begin{equation}  \label{Facteur_d_Importance}
\hat{P}_{n}(l):=\frac{\prod\limits_{i=1}^{n}p_{\mathbf{X}}\left(
Y_{i}(l)\right) }{g\left( Y_{1}^{n}(l)\right) }
\end{equation}
and where the $L$ samples $Y_{1}^{n}(l):=\left( Y_{1}(l),...,Y_{n}(l)\right) 
$ are i.i.d. with common density $g.$

The problem of finding a good sampling density $g$ has been widely explored
when $a_{n}=a$ is fixed and positive; this is the large deviation case; see
e.g. \cite{Bucklew2004}. The case when $a$ tends slowly to $E[u\left(\mathbf{%
X}\right)]$ from above (the moderate deviation case) is considered in \cite%
{Ermakov2007};

Under hypotheses to be recalled later, the \textit{classical} IS scheme
consists in the simulation of $n$ i.i.d. replications $%
Y_{1}^{(l)},...,Y_{n}^{(l)}$ with density $\pi^{a_{n}}$ on $\mathbb{R}$ and
therefore $g(y_{1},...,y_{n})=\pi^{a_{n}}(y_{1})...\pi^{a_{n}}(y_{n}).$ The
density $\pi^{a_{n}}$ is the so-called \textit{tilted} (or \textit{twisted})
density at point $a_{n}$ which, in case when $a_{n}=a$ is fixed, is called
the \textit{dominating point} of the set $(a,\infty);$ see \cite{Bucklew2004}%
. In spite of the fact that this terminology is usually used in the large
deviation case, we adopt it also in the moderate deviation one, for reasons
to be stated later on.

This approach produces efficient IS schemes, in the sense that the
computational burden necessary to obtain a relative precision of the
estimate with respect to $P_{n}$ does not grow exponentially as a function
of $n.$ It can be proved that in the large deviation range the variance of
the classical IS is proportional to $P_{n}^{2}\sqrt{n}$.

The numerator in the expression (\ref{Facteur_d_Importance}) is the product
of the $p_{\mathbf{X}_{1}}(Y_{i})$'s while the denominator need not be a
density of i.i.d. copies evaluated on the $Y_{i}^{\prime}s$. Indeed the
optimal choice for $g$ is the density of $\mathbf{X}_{1}^{n}:=\left( \mathbf{%
X}_{1},...,\mathbf{X}_{n}\right) $ conditioned upon $\left(\mathbf{X}%
_{1}^{n}\in\mathcal{E}_{n}\right)$, leading to a zero variance estimator. We
will propose an IS sampling density which approximates this conditional
density very sharply on its first components $y_{1},...,y_{k}$ where $k=k_{n}
$ is very large, namely $k/n\rightarrow 1.$ This motivates the title of this
paper.

Let us introduce a toy case in order to define the main step of the
procedure, namely the simulation of a sample under a proxy of the
conditional density. Assume $\mathbf{X}_{1}^{n}$ is a vector of $n$ i.i.d.
standard normal real valued random variables and $P_{n}:=P\left( \mathbf{S}%
_{1,n}>na\right) $ with $\mathbf{S}_{1,n}:=\mathbf{X}_{1}+...+\mathbf{X}_{n}$
and $a>0.$

1- For any $v>a$ the joint density $p_{nv}$ of $\mathbf{X}_{1},...\mathbf{X}%
_{n-1}$ conditionally upon $\left( \mathbf{S}_{1,n}=nv\right) $ is known
analytically and simulation under $p_{nv}$ is easy for any $v$. A general
form of this statement is Theorem 1, Section 2.

2-The optimal sampling density $g$ is similar to $p_{nv}$ with conditioning
event $\left( \mathbf{S}_{1,n}>na\right).$ The density $g$ is obtained
integrating $p_{nv}$ with respect to the the conditional distribution of $%
\mathbf{S}_{1,n}/n$ under $\left( \mathbf{S}_{1,n}>na\right) $ which is well
approximated by an exponential distribution on $\left( a,\infty \right) $
with expectation $a+(1/na).$ The corresponding general statement is Theorem
2 Section 2. Therefore samples under a proxy of $g$ are obtained through
Monte Carlo simulation as follows: draw $Y_{1}^{n}$ with density $p_{n%
\mathbf{V}}$ where $\mathbf{V}$ follows the just cited exponential density.
Insert these terms in (\ref{Facteur_d_Importance}) repeatedly to get $%
P_{g}^{(n)}.$ \bigskip

In the general case the joint distribution $p_{nv}$ cannot be approximated
sharply on the very long run $1,...,n-1$, but merely on $1,...,k_{n}$ with $%
k_{n}$ close to $n$. The approximation provided in Theorem \ref%
{thm:egal_chap2} and, as a consequence in Theorem \ref%
{Thm:approx_largeSets_chap2_a_fixed}, is valid on the first $k_{n}$
coordinates; a precise tuning of $k_{n}$ is provided in Section \ref%
{sec:howfar_chap2}. Since $v$ is simulated on the whole set $(a,+\infty)$,
no search is done in order to identify dominating points and no part of the
target set $(a,+\infty)$ is neglected in the simulation of runs; the example
in section \ref{sec:simu_chap2}, where the classical IS scheme is compared
to the present one, is illuminating in this respect. \bigskip

The merits of an IS estimator are captured through a number of criterions:

\begin{enumerate}
\item The asymptotic variance of the estimate

\item The stability of the Importance Factor

\item The hit rate of the IS scheme, which is the number of times the set $%
\mathcal{E}_{n}$ is reached by the simulated samples

\item Some run time indicator.
\end{enumerate}

Some mixed index have been proposed (see \cite{GlynnWhitt1992}), combining 1
and 4 with noticeable extension. The present paper provides an improvement
over classical IS schemes as measured by 1, 2, 3 here-above, as shown
numerically on some examples. These progresses are also argued on a
theoretical basis, following the quasi-optimality of the proposed IS scheme
resulting from the approximation of the conditional density. When the r.v. $%
\mathbf{U}_{i}$'s are real-valued, the present method might be costly. The
toy case which we present in the simulation study, pertaining to events $%
\left( |\mathbf{U}_{1,n}|>na_{n}\right) $ under $\mathbf{U}_{i}$'s proves
however that the \textit{observed} \textit{bias} of the estimate through IS
i.i.d. sampling can be important for reasonable $L$, which does not happen
with the present approach. Also the hit rate of the present proposal is
close to 100\%.

The criterion which we consider is different from the variance, and results
as an evaluation of the MSE of our estimate on specific subsets of the runs
generated by the sampling scheme, which we call \textit{typical subsets},
namely having probability going to $1$ under the sampling scheme as $n$
increases. On such sets, the MSE is proved to be of very small order with
respect to the variance of the classical estimate, which cannot be
diminished on any such typical subsets. It will be shown that the relative
gain in terms of simulation runs necessary to perform an $\alpha\%$ relative
error on $P_{n}$ drops by a factor $\sqrt{n-k}/\sqrt{n}$ with respect to the
classical IS scheme. Since $k$ is allowed to be close to $n$, the resulting
gain in variance is noticeable. Numerical evidence of this reduction in MSE
is produced. Also we present a way of choosing the value of $k_{n}$ with
respect to $n$ in such a way that the accuracy of the sampling scheme with
respect to the optimal one is somehow controlled. This rule is discussed
also numerically.

Alternative methods have been extensively developed for rare event
simulation (see \cite{BotevKroese2010} and references therein). The
splitting technique results in an adhoc covering $A_{1}\subset A_{2}\subset
...\subset A$. It is assumed that the conditional distribution $P_{k}$ of $%
\mathbf{U}_{1,n}$ given $\mathbf{U}_{1,n}\in nA_{k}$ is known. An ad hoc
choice of the $A_{k}$'s leading to a common value for the $P_{k}$'s provides
efficient estimator for $P_{n}$, with small run-times. However in the
present static case the calculation of the conditional distribution is
difficult, even in the real case, and requires a sharp asymptotic analysis
of large or moderate deviation probabilities.

It may seem that we could have reduced this paper to the case when $u$ is
the identity function, hence simulating runs $\mathbf{U}_{1}^{k}:=\left(
u\left( \mathbf{X}_{1}\right) ,...,u\left( \mathbf{X}_{k}\right) \right) $
under $\left(\mathbf{U}_{1,n}>na\right).$ However it often occurs that the
conditioning event is defined through a joint set of conditions, say 
\begin{equation}  \label{lin_constraint}
u\left( \mathbf{X}_{1}\right) +...+u\left( \mathbf{X}_{n}\right) >na
\end{equation}
and 
\begin{equation}
h\left( \mathbf{X}_{1}^{n}\right) \in B_{n}  \label{g_constraint}
\end{equation}%
for some function $h$ and some measurable set $B_{n}.$ Clearly in most cases
the approximation of the density of $\mathbf{X}_{1}^{k}$ under both
constraints is intractable and the approximation of the density of $\mathbf{X%
}_{1}^{k}$ conditioned upon $\left(\mathbf{X}_{1}^{n}\in\mathcal{E}%
_{n}\right)$ provides a good IS sampling scheme for the estimation of 
\begin{equation*}
P\left( u\left( \mathbf{X}_{1}\right) +...+u\left( \mathbf{X}_{n}\right) >na
\cap h\left( \mathbf{X}_{1}^{n}\right) \in B_{n}\right) .
\end{equation*}%
A simple example is when the constraint writes 
\begin{equation*}
\mathbf{X}_{1}^{n}\in D_{n}
\end{equation*}%
and $D_{n}$ is included in a set defined through (\ref{lin_constraint}). The
function $u$ and the value of $a$ may be fitted such that (\ref%
{lin_constraint}) makes minimal the difference 
\begin{align*}
& P\left(u\left(\mathbf{X}_{1}\right)+...+u\left(\mathbf{X}%
_{n}\right)>na\right) \\
& -P\left( \mathbf{X}_{1}^{n}\in D_{n}\right) .
\end{align*}

Our proposal therefore hinges on the local approximation of the conditional
distribution of longs runs $\mathbf{X}_{1}^{k}$ from $\mathbf{X}_{1}^{n}.$
This cannot be achieved through the classical theory of large deviations,
nor through the moderate deviations one, first developed by \cite%
{deAcosta1992} and more recently by \cite{Ermakov2007}. At the contrary the
ad hoc procedure developed in the range of large deviations by \cite%
{DiaconisFreedman1988} for the local approximation of the conditional
distribution of $\mathbf{X}_{1}^{k}$ given the value of $\left(\mathbf{S}%
_{1,n}:=\mathbf{X}_{1}+...+\mathbf{X}_{n}\right)$ is the starting point of
the method leading to the present approach. We rely on \cite{BroniatowskiCaron2011} where the basic approximation used in the present paper can be found. A first draft in the direction of the present work is in \cite{BroniatowskiRitov2009}.

The present approach can be extended to the case of a multivariate
constraint for a multidimensional problem, i.e. when for all $x$ in $\mathbb{%
R}^{d}$, $u\left( x\right)$ and $a$ are $\mathbb{R}^{s}$ -valued. This will
not be considered here.

\subsection{Notations and Assumptions}

The following notation and assumptions are kept throughout the paper without
further reference.

\subsubsection{Conditional densities and their approximations}

Throughout the paper the value of a density $p_{\mathbf{Z}}$ of some
continuous random vector $\mathbf{Z}$ at point $z$ may be written $p_{%
\mathbf{Z}}(z)$ or $p\left( \mathbf{Z}=z\right) ,$ which may prove more
convenient according to the context. The normal density function on $\mathbb{%
R}$ with mean $\mu $ and variance $\tau $ at $x$ is denoted $\mathfrak{n}%
\left( \mu ,\tau ,x\right) .$

Let $p_{nv}$ denote the density of $\mathbf{X}_{1}^{k}$ under the local
condition $\left(\mathbf{U}_{1,n}=nv\right)$ 
\begin{equation}  \label{def:p_cond_ponc_chap2}
p_{nv}\left( \mathbf{X}_{1}^{k}=Y_{1}^{k}\right) :=p(\left. \mathbf{X}%
_{1}^{k}=Y_{1}^{k}\right\vert \mathbf{U}_{1,n}=nv)
\end{equation}
where $Y_{1}^{k}$ belongs to $\mathbb{R}^{k}.$

We will also consider the density $p_{nA}$ of $\mathbf{X}_{1}^{k}$
conditioned upon $\left( \mathbf{U}_{1,n}>na\right) $ 
\begin{equation}  \label{def:p_cond_plusgrand_chap2}
p_{nA}\left( \mathbf{X}_{1}^{k}=Y_{1}^{k}\right) :=p(\left. \mathbf{X}%
_{1}^{k}=Y_{1}^{k}\right\vert\mathbf{U}_{1,n}>na).
\end{equation}

The approximating density of $p_{nv}$ is denoted $g_{nv}$; the corresponding
approximation of $p_{nA}$ is denoted $g_{nA}.$ Explicit formulas for those
densities are presented in the next section.

\subsubsection{Tilted densities and related quantities}

The real valued measurable function $u$ is assumed to be unbounded; standard
transformations show that this assumption is not restrictive. It is assumed
that $\mathbf{U}=u\left(\mathbf{X}\right)$ has a density $p_{\mathbf{U}}$
w.r.t. the Lebesgue measure on $\mathbb{R}.$ We also assume that the
characteristic function of the random variable $\mathbf{U}$ is assumed to
belong to $L^{r}$ for some $r\geq{1}.$

The r.v. $\mathbf{U}$ is supposed to fulfill the Cramer condition: its
moment generating function satisfies 
\begin{equation*}
\phi_{\mathbf{U}}(t):=E\exp t\mathbf{U}<\infty
\end{equation*}
for $t$ in a non void neighborhood of $0.$ Define the functions $%
m(t),s^{2}(t)$ and $\mu_{3}(t)$ as the first, second and third derivatives
of $\log \phi_{\mathbf{U}}(t)$, and $m^{-1}$ denote the reciprocal function
of $m.$

Denote 
\begin{equation}  \label{def:tilted_U_chap2}
\pi _{\mathbf{U}}^{\alpha }(u):=\frac{\exp tu}{\phi _{\mathbf{U}}(t)}p_{%
\mathbf{U}}\left( u\right)
\end{equation}%
with $m(t)=\alpha $ and $\alpha $ belongs to the support of $P_{\mathbf{U}},$
the distribution of $\mathbf{U}.$ The density $\pi_{\mathbf{U}}^{\alpha}$ is
the \textit{tilted }density with parameter $\alpha.$ Also it is assumed that
this latest definition of $t$ makes sense for all $\alpha$ in the support of 
$\mathbf{U}.$ Conditions on $\phi_{\mathbf{U}}(t)$ which ensure this fact
are referred to as \textit{steepness properties}, and are exposed in \cite%
{barndorff-Nielsen1978}, p153.

We also introduce the family of densities 
\begin{equation}  \label{def:tilted_u(X)_chap2}
\pi _{u}^{\alpha }(x):=\frac{\exp tu(x)}{\phi _{\mathbf{U}}(t)}p_{\mathbf{X}%
}\left( x\right) .
\end{equation}
with $\Pi_{u}^{\alpha}$ the associated distribution.

\subsubsection{Specific sequences}

The sequence $a_{n}$ is introduced in the paper. For notational convenience
its current terms will be denoted $a$ without referring to the subscript $n.$

\section{Conditioned samples}

\label{sec:cond_chap2}

The starting point is the approximation of $p_{nv}$ defined in (\ref%
{def:p_cond_ponc_chap2}) on $\mathbb{R}^{k}$ for large values of $k$ under
the point condition 
\begin{equation*}
\left(\mathbf{U}_{1,n}=nv\right)
\end{equation*}
when $v$ belongs to $\left( a,\infty\right) .$ We refer to \cite{BroniatowskiCaron2011} for this result.

We introduce a positive sequence $\epsilon_{n}$\ which satisfies

\begin{align}  \label{E1_chap2}
\lim_{n\rightarrow\infty}\epsilon_{n}\sqrt{n-k} & =\infty  \tag{E1} \\
\lim_{n\rightarrow\infty}\epsilon_{n}\left( \log n\right) ^{2} & =0.\text{ }
\tag{E2}  \label{E2_chap2}
\end{align}

Define a density $g_{nv}(y_{1}^{k})$ on $\mathbb{R}^{k}$ as follows. Set 
\begin{equation}  \label{def:g_0_chap2}
g_{0}(\left. y_{1}\right\vert y_{0}):=\pi_{u}^{v}(y_{1})
\end{equation}
with $y_{0}$ arbitrary and, for $1\leq i\leq k-1$, define $g(\left.
y_{i+1}\right\vert y_{1}^{i})$ recursively.

Set $t_{i}$ the unique solution of the equation 
\begin{equation}  \label{def:m_i_chap2}
m_{i}:=m(t_{i})=\frac{n}{n-i}\left( v-\frac{u_{1,i}}{n}\right)
\end{equation}
where $u_{1,i}:=u(y_{1})+...+u(y_{i}).$

Define 
\begin{equation}  \label{def:g_i_chap2}
g(\left. y_{i+1}\right\vert y_{1}^{i})=C_{i}p_{\mathbf{X}}(y_{i+1})\mathfrak{%
n}\left( \alpha\beta+v,\alpha,u(y_{i+1})\right)
\end{equation}
where $C_{i}$ is a normalizing constant. Here 
\begin{equation}  \label{def:alpha_chap2}
\alpha=s^{2}(t_{i})\left( n-i-1\right)
\end{equation}%
\begin{equation}  \label{def:beta_chap2}
\beta=t_{i}+\frac{\mu_{3}\left( t_{i}\right) }{2s^{4}(t_{i})\left(
n-i-1\right) }.
\end{equation}

Set 
\begin{equation}  \label{def:g_nv_chap2}
g_{nv}\left( y_{1}^{k}\right) :=g_{0}(\left. y_{1}\right\vert
y_{0})\prod\limits_{i=1}^{k-1}g(\left. y_{i+1}\right\vert y_{1}^{i}).
\end{equation}

\begin{theorem}
\label{thm:egal_chap2} Assume (\ref{E1_chap2}) and (\ref{E2_chap2}). Then
(i) 
\begin{equation}  \label{thm:egal_(i)_chap2}
p_{nv}\left( \mathbf{X}_{1}^{k}=Y_{1}^{k}\right)
=g_{nv}(Y_{1}^{k})(1+o_{P_{nv}}(\epsilon_{n}\left(\log n\right)^{2}))
\end{equation}
and (ii) 
\begin{equation}  \label{thm:egal_(ii)_chap2}
p_{nv}\left( \mathbf{X}_{1}^{k}=Y_{1}^{k}\right)
=g_{nv}(Y_{1}^{k})(1+o_{G_{nv}}(\epsilon_{n}\left(\log n\right)^{2})).
\end{equation}
\end{theorem}

The approximation stated in the above statement (i) holds on \textit{typical
paths} generated under the conditional density $p_{ns}$; in the same way,
statement (ii) holds under the sampling scheme $g_{ns}.$ Therefore they do
not hold on the entire space $\mathbb{R}^{k}$ which would require more
restrictive hypotheses on the characteristic function of $u\left( \mathbf{X}%
_{1}\right) $; see \cite{DiaconisFreedman1988} for such conditions in the
case when $k$ is allowed to grow slowly with respect to $n$ and $a$ is
fixed. However the above theorem provides optimal approximations on the
entire space $\mathbb{R}^{k}$ for all $k$ between $1$ and $n-1$ in the
gaussian case and $u(x)=x$, since $g_{ns}\left( y_{1}^{k}\right) $ coincides
with the conditional density. As stated in \cite{BroniatowskiCaron2011}, the extension of our results from typical paths to the whole space $\mathbb{R}^{k}$ holds: convergence of the relative error on large sets imply that the total variation distance between the conditioned measure and its approximation goes to $0$ on the entire space. So our results provide an extension of \cite{DiaconisFreedman1988} and \cite{DemboZeitouni1996} who considered the case when $k$ is of small order with respect to $n;$ the conditions which are assumed in the present paper are weaker than those assumed in the just cited works; however, in contrast with their results, we do not provide explicit rates for the convergence to $0$ of the total variation distance on $\mathbb{R}^{k}$.

As stated above the optimal choice for the sampling density is $p_{nA}$ for
which we state an approximation result, extending Theorem \ref%
{thm:egal_chap2}.

We state the approximating density for $p_{nA}$ defined in (\ref%
{def:p_cond_plusgrand_chap2}). It holds 
\begin{equation}
p_{nA}(x_{1}^{k})=\int_{a}^{\infty }p_{nv}\left( \mathbf{X}%
_{1}^{k}=x_{1}^{k}\right) p(\left. \mathbf{U}_{1,n}/n=v\right\vert \mathbf{U}%
_{1,n}>na)dv  \label{etoile_chap2}
\end{equation}%
so that, in contrast with the classical IS approach for this problem we do
not consider the dominating point approach but merely realize a sharp
approximation of the integrand at any point of the domain $\left( a,\infty
\right) $ and consider the dominating contribution of all those
distributions in the evaluation of the conditional density $p_{nA}.$ A
similar point of view has been considered in \cite{BarbeBroniatowski2004}
for sharp approximations of Laplace type integrals in $\mathbb{R}^{d}.$

The approximation of $p_{nA}$ is handled on some small interval $\left(
a,a+c\right) $, thus on the principal part of this integral.

Let $c_{n}$ denote a positive sequence such that (C) 

\begin{center}
$%
\begin{array}{c}
\lim_{n\rightarrow \infty }nc_{n}m^{-1}(a)=\infty  \\ 
\displaystyle{\sup_{n\geq 1}\frac{nc_{n}}{(n-k)}<\infty} 
\end{array}
$
\end{center}

and denote $c$ the current term of the sequence $c_{n}$.

Denote (A) the following set of conditions
\begin{eqnarray*}
\lim_{n\rightarrow\infty}(n-k)\left(m^{-1}\left(a\right)\right)^{2}= \infty \\
\lim_{n\rightarrow\infty}\frac{m^{-1}\left(a\right)}{\epsilon_{n}}=\infty
\end{eqnarray*}
which trivially holds when $\lim_{n\rightarrow \infty }a_{n}>E\mathbf{U}.$

Define on $\mathbb{R}^{k}$ the density 
\begin{align}
& g_{nA}(y_{1}^{k})  \label{def:g_nA_chap2} \\
& :=\frac{nm^{-1}\left( a\right) \int_{a}^{a+c}g_{nv}(y_{1}^{k})\left( \exp
-nm^{-1}\left( a\right) (v-a)\right) dv}{1-\exp -nm^{-1}\left( a\right) c}. 
\notag
\end{align}%
The density

\begin{equation}
\frac{nm^{-1}\left( a\right) \left( \exp -nm^{-1}\left( a\right)
(v-a)\right) {\mathds{1}}_{\left( a,a+c\right) }(v)}{1-\exp -nm^{-1}\left(
a\right) c}  \label{condS_chap2}
\end{equation}%
which appears in (\ref{def:g_nA_chap2}) approximates $p(\left. \mathbf{U}%
_{1,n}/n=v\right\vert a<\mathbf{U}_{1,n}/n<a+c).$

The \textit{variance function} $V$ of the distribution of $\mathbf{U}$ is
defined on the span of $\mathbf{U}$ through 
\begin{equation*}
v\rightarrow V(v):=s^{2}(m^{-1}(v))
\end{equation*}%
Denote (V) the condition%
\begin{equation} \label{cond:V}
\displaystyle{\sup_{n\geq 1}\left( \sqrt{n}m^{-1}(a)\int_{a}^{\infty
}V^{\prime }(v)\left( \exp -nm^{-1}(a)\left( v-a\right) \right) dv\right) 
<\infty .} \tag{(V)}
\end{equation}

\begin{theorem}
\label{Thm:approx_largeSets_chap2_a_fixed} Assume (A), (C), (\ref{cond:V}), (\ref{E1_chap2}) and (\ref{E2_chap2}).. Then (i) 
\begin{equation}
p_{nA}\left( \mathbf{X}_{1}^{k}=Y_{1}^{k}\right)
=g_{nA}(Y_{1}^{k})(1+o_{P_{nA}}(\delta _{n}))
\label{Thm:approx_largeSets(i)_chap2_a_fixed}
\end{equation}%
and (ii) 
\begin{equation}
p_{nA}\left( \mathbf{X}_{1}^{k}=Y_{1}^{k}\right)
=g_{nA}(Y_{1}^{k})(1+o_{G_{nA}}(\delta _{n}))
\label{Thm:approx_largeSets(ii)_chap2_a_fixed}
\end{equation}
\end{theorem}

where 
\begin{equation}
\delta _{n}:=\max \left( \epsilon _{n}\left( \log n\right) ^{2},\left( \exp
\left( -ncm^{-1}(a)\right) \right) ^{\delta }\right) .
\label{vitesse_chap2_a_fixed}
\end{equation}

for any positive $\delta <1$.

The proof of Theorem \ref{Thm:approx_largeSets_chap2_a_fixed} is deferred to
the Appendix.

\begin{remark}
Most distributions used in statistics satisfy (V); numerous papers have
focused on the properties of variance functions and classification of
distributions. see e.g.\cite{LetacMora} and references therein.
\end{remark}

\begin{remark}
When $a$ is fixed, the set of conditions (A) hold. In the case where $a=a_{n}$ converges to $E\mathbf{U}$, the set of conditions (A) does not cover the
CLT zone. Indeed, the first condition of (A) implies that $m^{-1}(a)$
satisfies, for some $\delta >0$, 
\begin{equation*}
|m^{-1}(a)n^{1/2+\delta }|<\infty .
\end{equation*}%
Besides this limitation, choosing $k$ and $\epsilon _{n}$ according to (A), (C), (\ref{E1_chap2}) and (\ref{E2_chap2}) is always possible. More $a_{n}$
convergences slowly to $E\mathbf{U}$, more $k$ can be choosen large with
respect to $n.$
\end{remark}

\section{How far is the approximation valid?}

\label{sec:howfar_chap2}

This section provides a rule leading to an effective choice of the crucial
parameter $k=k_{n}$ in order to achieve a given accuracy bound for the
relative error committed substituting $p_{nA}$ by $g_{nA}$. The largest $k$
the best the estimate of the rare event probability. We consider the large
deviation case, assuming $a$ fixed.

The accuracy of the approximation is measured through 
\begin{equation}
ERE(k):=E_{G_{nA}}\left( 1_{D_{k}}\left( Y_{1}^{k}\right) \frac{p_{nA}\left(
Y_{1}^{k}\right) -g_{nA}\left( Y_{1}^{k}\right) }{p_{nA}\left(
Y_{1}^{k}\right) }\right)
\end{equation}
and 
\begin{equation}
VRE(k):=Var_{G_{nA}}\left( 1_{D_{k}}\left( Y_{1}^{k}\right) \frac{%
p_{nA}\left( Y_{1}^{k}\right) -g_{nA}\left( Y_{1}^{k}\right) }{p_{nA}\left(
Y_{1}^{k}\right) }\right)  \label{RE}
\end{equation}%
respectively the expectation and the variance of the relative error of the
approximating scheme when evaluated on 
\begin{equation*}
D_{k}:=\left\{ y_{1}^{k}\in \mathbb{R}^{k}\text{ such that }\left\vert
g_{u_{1,n}}(y_{1}^{k})/p_{u_{1,n}}\left( y_{1}^{k}\right) -1\right\vert
<\delta _{n}\right\}
\end{equation*}
with $\epsilon _{n}\left( \log n\right) ^{2}/\delta _{n}\rightarrow 0$ and $%
\delta _{n}\rightarrow 0;$ therefore $G_{u_{1,n}}\left( D_{k}\right)
\rightarrow 1.$ The r.v$^{\prime }$s $Y_{1}^{k}$ are sampled under $g_{nA}.$
Note that the density $p_{nA}$ is usually unknown. The argument is somehow
heuristic and informal; nevertheless the rule is simple to implement and
provides good results. We assume that the set $D_{k}$ can be substituted by $%
\mathbb{R}^{k}$ in the above formulas, therefore assuming that the relative
error has bounded variance, which would require quite a lot of work to be
proved under appropriate conditions, but which seems to hold, at least in
all cases considered by the authors. We keep the above notation omitting
therefore any reference to $D_{k}$ .

Consider a two-sigma confidence bound for the relative accuracy for a given $%
k$, defining 
\begin{equation*}
CI(k):=
\end{equation*}
\begin{equation*}
\left[ ERE(k)-2\sqrt{VRE(k)},ERE(k)+2\sqrt{VRE(k)}\right] .
\end{equation*}
\bigskip

Let $\delta$ denote an acceptance level for the relative accuracy. Accept $k$
until $\delta$ belongs to $CI(k).$ For such $k$ the relative accuracy is
certified up to the level $5\%$ roughly.

In \cite{BroniatowskiCaron2011}, a similar question is addressed and a proxy of the curve $%
\delta\rightarrow k_{\delta}$ is provided in order to define the maximal $k$
leading to a given relative accuracy under the point condition $\left(%
\mathbf{U}_{1,n}=na\right),$ namely when $p_{nA}$ is replaced by $p_{na}$
and $g_{nA}$ by $g_{na}.$

Consider the ratio $g_{nA}(Y_{1}^{k})/p_{nA}\left( Y_{1}^{k}\right) $ and
use Cauchy's mean value theorem to obtain 
\begin{equation*}
g_{nA}(Y_{1}^{k})/p_{nA}\left( Y_{1}^{k}\right)
\end{equation*}
\begin{equation*}
=\frac{\int_{a}^{a+c}g_{nv}(\mathbf{X}_{1}^{k}=Y_{1}^{k})\left( \exp
-nm^{-1}\left( a\right) \left( v-a\right) \right) dv}{\int_{a}^{a+c}p_{nv}%
\left( \mathbf{X}_{1}^{k}=Y_{1}^{k}\right) \left( \exp-nm^{-1}\left(
a\right) \left( v-a\right) \right) ds}
\end{equation*}
\begin{equation*}
\left(1+o_{G_{nA}}\left( 1\right)\right)
\end{equation*}
\begin{equation*}
=\frac{g_{n\alpha}(\mathbf{X}_{1}^{k}=Y_{1}^{k})}{p_{n\alpha}\left( \mathbf{X%
}_{1}^{k}=Y_{1}^{k}\right) }\left( 1+o_{G_{nA}}\left( 1\right) \right)
\end{equation*}
for some $\alpha$ between $a$ and $a+c.$ Since $a$ and $c$ are fixed,
eventually small, it is reasonable to substitute $\alpha$ by $a$ in order to
evaluate the accuracy of the approximation. We thus inherit of the relative
efficiency curve in \cite{BroniatowskiCaron2011}, to which we refer for definition and
derivation.

We briefly state the necessary steps required for the calculation of the
graph of a proxy of $k\rightarrow CI(k).$

Introduce 
\begin{equation*}
D:=\left[\frac{\pi_{\mathbf{U}}^{a}(a)}{p_{\mathbf{U}}(a)}\right] ^{n}
\end{equation*}
and 
\begin{equation*}
N:=\left[ \frac{\pi_{\mathbf{U}}^{m_{k}}\left( m_{k}\right) }{p_{\mathbf{U}%
}\left( m_{k}\right) }\right] ^{\left( n-k\right) }
\end{equation*}
with $m_{k}$ defined in (\ref{def:m_i_chap2}). Define $t$ by $m(t)=a$ and $%
t^{k}$ by $m(t^{k})=m_{k}.$ Define 
\begin{equation}  \label{A(l)_chap2}
A\left( Y_{1}^{k}\right) :=\frac{n-k}{n}\left( \frac{g_{nA}\left(
Y_{1}^{k}\right) }{p_{\mathbf{X}}\left( Y_{1}^{k}\right) }\right) ^{3}\left( 
\frac{N}{D}\right) ^{2}\frac{s^{2}(t^{k})}{s^{2}(t)}.
\end{equation}

Simulate $L$ i.i.d. samples $Y_{1}^{k}(l)$ , each one made of $k$ i.i.d.
replications under $p_{\mathbf{X}}$; set 
\begin{equation*}
\widehat{A}:=\frac{1}{L}\sum_{l=1}^{L}A\left( Y_{1}^{k}(l)\right) .
\end{equation*}
We use the same approximation for $B.$ Define%
\begin{equation}
B\left( Y_{1}^{k}\right) :=\sqrt{\frac{n-k}{n}}\left( \frac{g_{nA}\left(
Y_{1}^{k}\right) }{p_{\mathbf{X}}\left( Y_{1}^{k}\right) }\right) ^{2}\left( 
\frac{N}{D}\right) \frac{s^{2}(t^{k})}{s^{2}(t)}  \label{B(l)_chap2}
\end{equation}
and 
\begin{equation*}
\widehat{B}:=\frac{1}{L}\sum_{l=1}^{L}B\left( Y_{1}^{k}(l)\right)
\end{equation*}
with the same $Y_{1}^{k}(l)^{\prime}s$ as above.

Set 
\begin{equation*}
\overline{VRE}(k):=\widehat{A}-\widehat{B}^{2}.
\end{equation*}
which is a fair approximation of $VRE(k).$

In the same way a proxy for $ERE$ is defined through 
\begin{equation*}
\overline{ERE}(k):=1-\widehat{B}.
\end{equation*}
A proxy of $CI(k)$ can now be defined through 
\begin{align*}
\overline{CI}(k):=
\end{align*}
\begin{equation}  \label{CIbarre_chap2}
\left[ \overline{ERE}(k)-2\sqrt{\overline{VRE}(k)},\overline{ERE}(k)+2\sqrt{%
\overline{VRE}(k)}\right] .
\end{equation}

We now check the validity of the just above approximation, comparing $%
\overline{CI}(k)$ with $CI(k)$ on a toy case. Detailed algorithms leading to
effective procedures are exposed in the next section.

Consider $u(x)=x.$ The case when $p_{\mathbf{X}}$ is a centered exponential
distribution with variance $1$ allows for an explicit evaluation of $CI(k)$
making no use of Lemma \ref{Lemma_Jensen_chap2}. The conditional density $%
p_{nv}$ is calculated analytically, the density $g_{nv}$ is obtained through
(\ref{def:g_nv_chap2}), hence providing a benchmark for our proposal. The
terms $\widehat{A}$ and $\widehat{B}$ are obtained by Monte Carlo simulation
following the algorithm presented hereunder. Tables 1,2 and 3,4 show the
increase in $\delta $ w.r.t. $k$ in the large deviation range, with $a$ such
that $P_{n}$ $:=P\left( \mathbf{S}_{1,n}>na\right) \simeq 10^{-8}.$ We have
considered two cases, when $n=100$ and when $n=1000.$ These tables show that
the approximation scheme is quite accurate, since the relative error is
fairly small even in very high dimension spaces. Also they show that $%
\overline{ERE}$ et $\overline{CI}$ provide good tools for the assessing the
value of $k.$ Denote $P_{n}:=P\left( \mathbf{S}_{1,n}>na\right).$

\begin{figure}[!ht]
\centering \includegraphics*[scale=0.4]{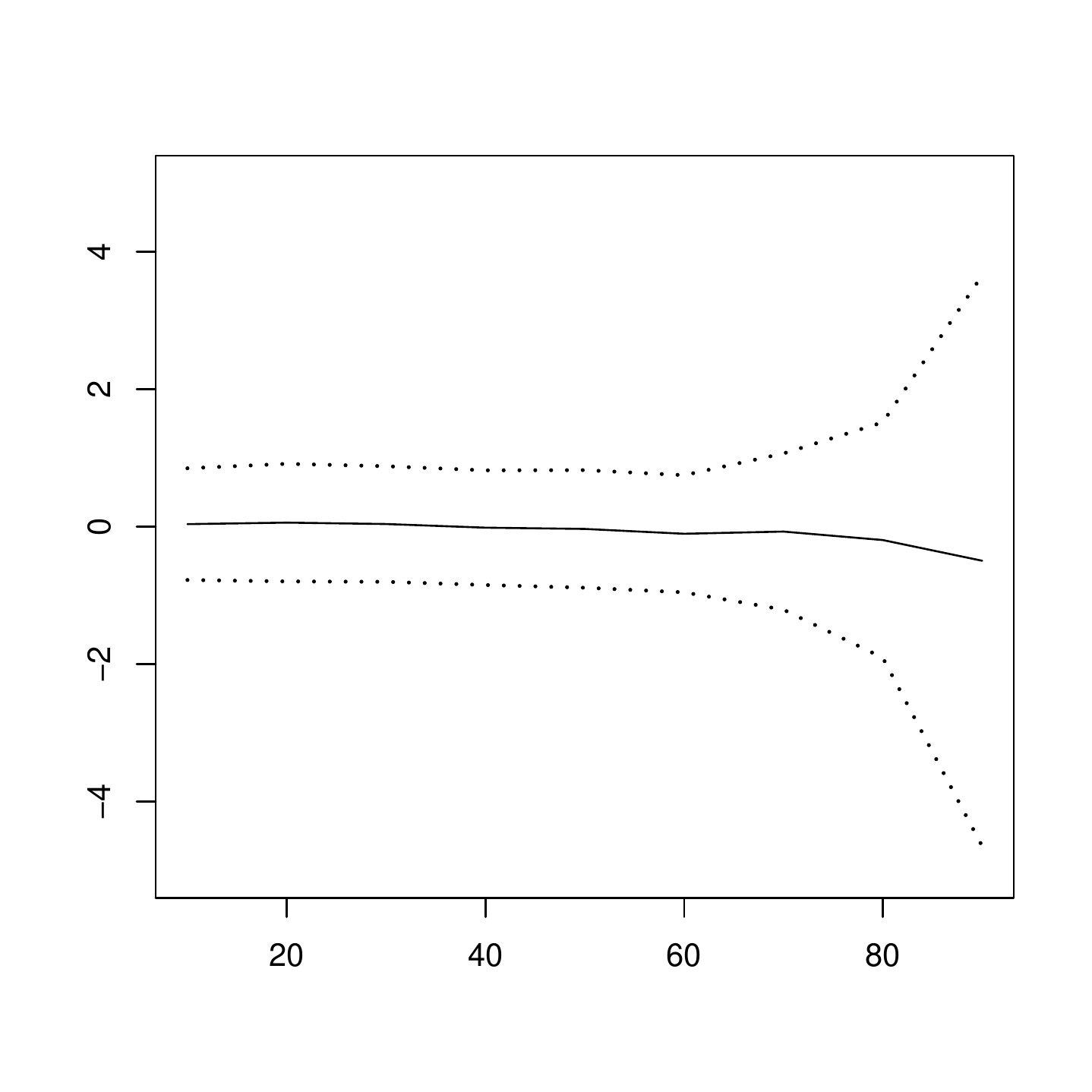}
\caption{$\overline{ERE}(k)$(solid line) along with upper and lower bound of 
$\overline{CI}(k)$(dotted line) as a function of $k$ with $n=100$ and $a$
such that $P_{n}\simeq10^{-8}.$}
\end{figure}

\begin{figure}[!ht]
\centering \includegraphics*[ scale=0.4]{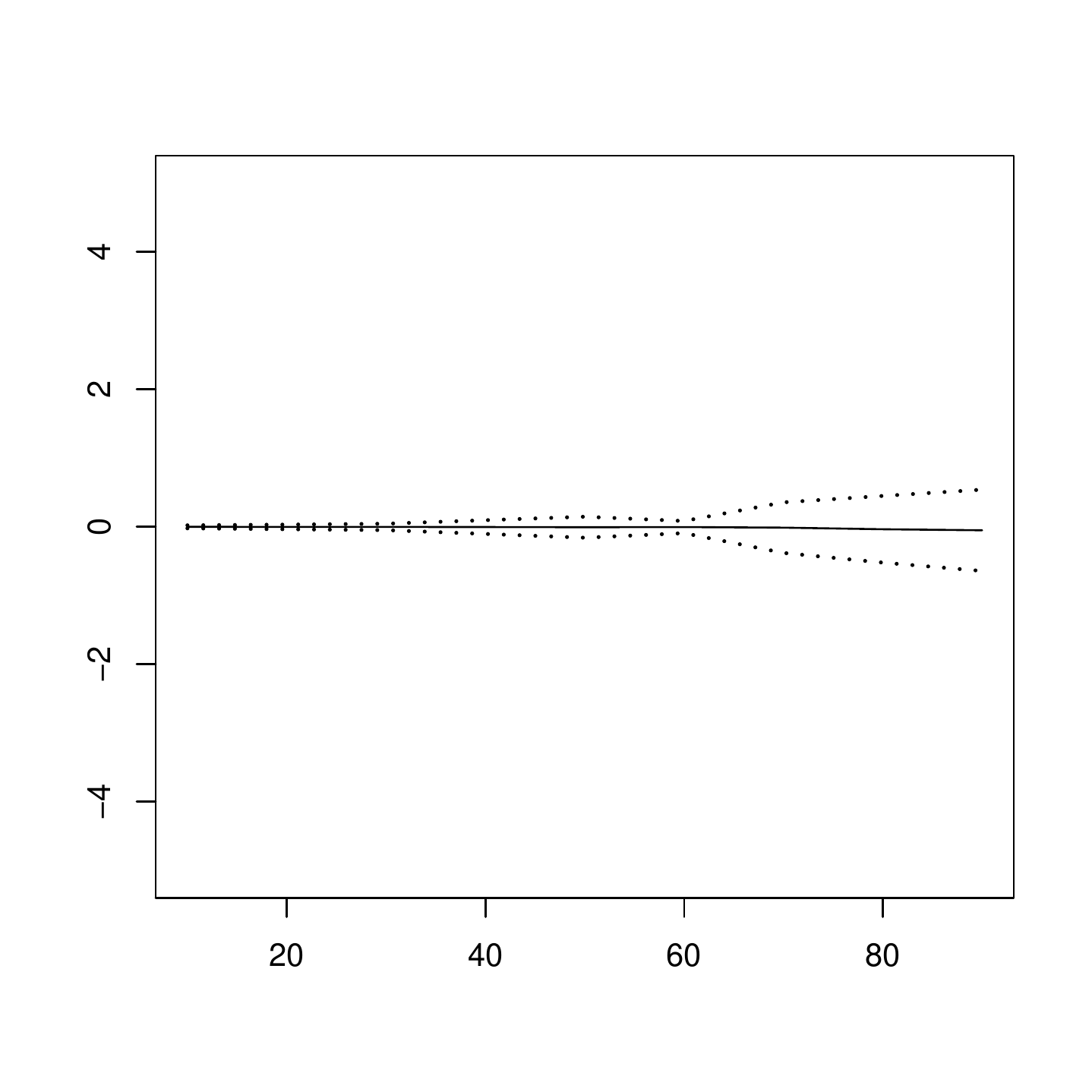}
\caption{$ERE(k)$(solid line) along with upper and lower bound of $CI(k)$%
(dotted line) as a function of $k$ with $n=100$ and $a$ such that $%
P_{n}\simeq10^{-8}.$}
\end{figure}

\begin{figure}[!ht]
\centering \includegraphics*[ scale=0.4]{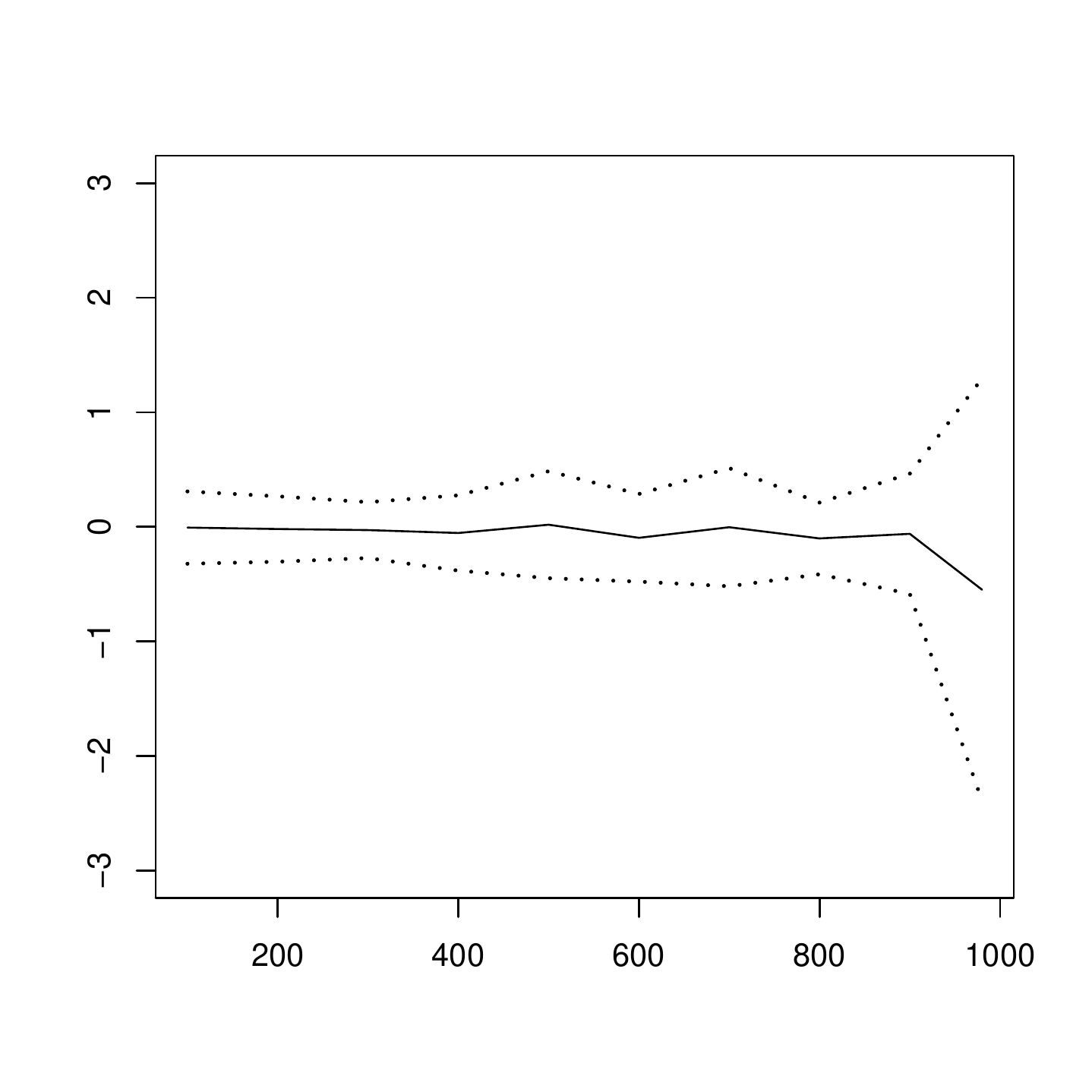}
\caption{$\overline{ERE}(k)$(solid line) along with upper and lower bound of 
$\overline{CI}(k)$(dotted line) as a function of $k$ with $n=1000$ and $a$
such that $P_{n}\simeq10^{-8}.$}
\end{figure}

\begin{figure}[!ht]
\centering \includegraphics*[ scale=0.4]{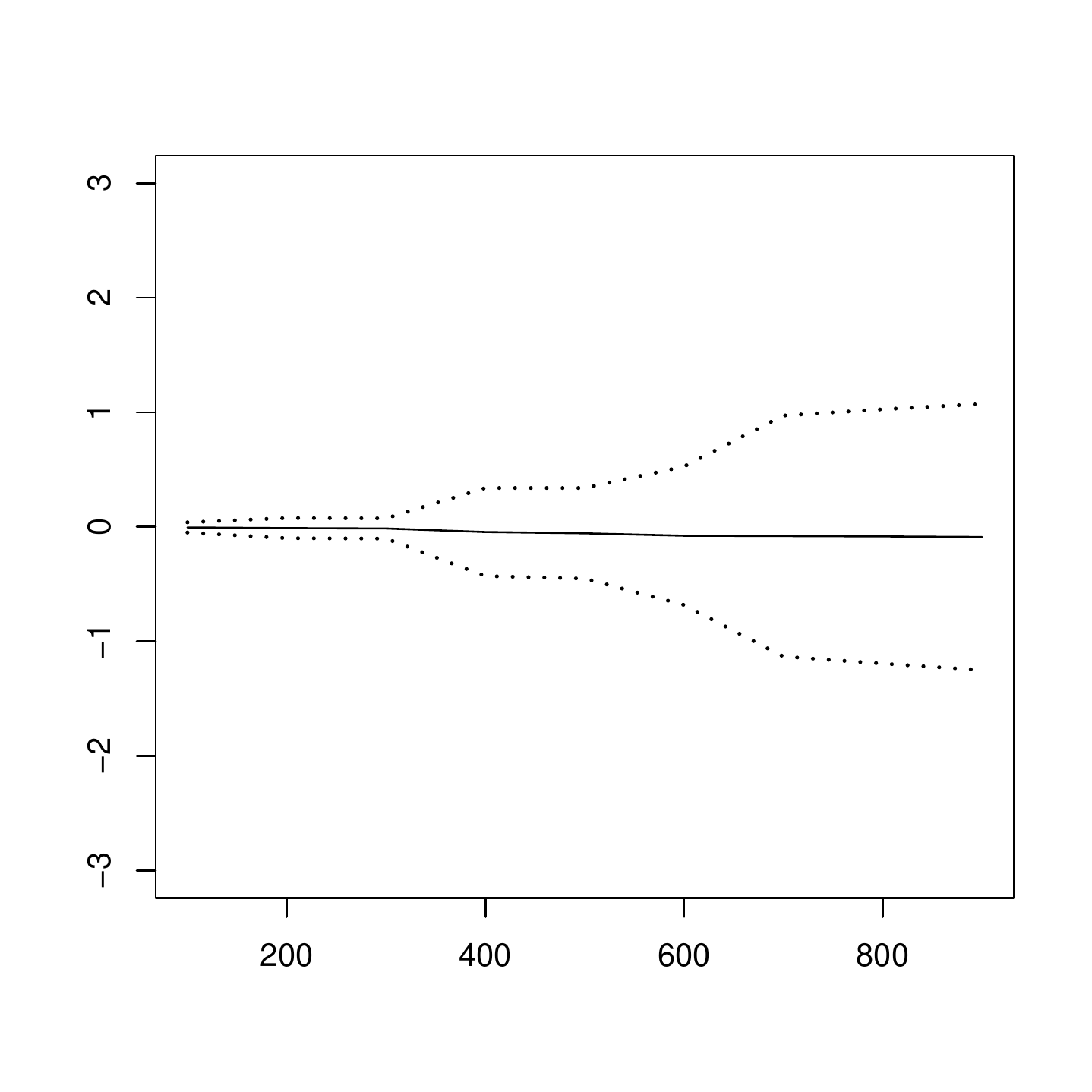}
\caption{$ERE(k)$(solid line) along with upper and lower bound of $CI(k)$%
(dotted line) as a function of $k$ with $n=1000$ and $a$ such that $%
P_{n}\simeq10^{-8}.$}
\end{figure}

\section{The new Estimator and the algorithms}

\label{sec:implemen_chap2}

\subsection{Adaptive IS Estimator for rare event probability}

The IS scheme produces samples $Y:=\left( Y_{1},...,Y_{k}\right) $
distributed under $g_{nA}$, which is a continuous mixture of densities $%
g_{nv}$ as in (\ref{def:g_nv_chap2}), with exponential mixing measure with
parameter $nm^{-1}\left( a\right) $ on $\left( a,\infty\right)$ 
\begin{equation}  \label{dens_exp_s}
\mathds{1}_{\left( a,\infty\right) }(v)nm^{-1}\left( a\right) \exp\left[
-nm^{-1}\left( a\right) \left( v-a\right) \right]
\end{equation}

Since all IS schemes produce unbiased estimators, and since the truncation
parameter $c$ in (\ref{def:g_nA_chap2}) is immaterial, we consider
untruncated versions of $g_{nA}$ defined in (\ref{def:g_nA_chap2})
integrating on $(a,\infty)$instead of $\left( a,a+c\right) .$ This avoids a
number of computational and programming questions, a difficult choice of an
extra parameter $c$, and does not change the numerical results; this point
has been checked carefully by the authors. Wee keep the notation $g_{nA}$
for the untruncated density.

The density $g_{nA}$ is extended from $\mathbb{R}^{k}$ onto $\mathbb{R}^{n}$
completing the $n-k$ remaining coordinates with i.i.d. copies of r.v's $%
Y_{k+1},...,Y_{n}$ with common tilted density 
\begin{equation}
g_{nA}\left( \left. y_{k+1}^{n}\right\vert y_{1}^{k}\right) :=\prod
\limits_{i=k+1}^{n}\pi_{u}^{m_{k}}(y_{i})  \label{complementk,n}
\end{equation}
with $m_{k}:=m(t^{k})=\frac{n}{n-k}\left( v-\frac{u_{1,k}}{n}\right) $ as in
(\ref{def:m_i_chap2}) and 
\begin{equation*}
u_{1,k}=\sum_{i=1}^{k}u(y_{i})
\end{equation*}

The last $n-k$ r.v's $\mathbf{Y}_{i}$'s are therefore drawn according to the
classical i.i.d. scheme in phase with \cite{SadowskyBucklew1990} or \cite%
{Ermakov2007} schemes in the large or moderate deviation setting.

We now define our IS estimator of $P_{n}:=P\left( \mathbf{U}_{1,n}>na\right).
$

Let $Y_{1}^{n}(l):=Y_{1}(l),...,Y_{n}(l)$ be generated under $g_{nA}.$ Let 
\begin{equation}  \label{Pchapl}
\widehat{P_{n}}(l):=\frac{\prod_{i=0}^{n}p_{\mathbf{X}}(Y_{i}(l))}{%
g_{nA}(Y_{1}^{n}(l))}\mathds{1}_{\mathcal{E}_{n}}\left( Y_{1}^{n}(l)\right)
\end{equation}
and define 
\begin{equation}  \label{Pchap}
\widehat{P_{n}}:=\frac{1}{L}\sum_{l=1}^{L}\widehat{P_{n}}(l).
\end{equation}
in accordance with (\ref{FORM_IS}).

\subsection{Algorithms}

\begin{center}
First, we present a series of three algorithms (Algorithms \ref%
{algo_g_v_chap2}, \ref{algo_g_A_chap2} and \ref{algo_calcul_k_chap2}) which
produces the curve $k\rightarrow\overline{RE}(k).$ The resulting $%
k=k_{\delta }$ is the longest size of the runs which makes $g_{nA}$ a good
proxy for $p_{nA}.$
\end{center}

\IncMargin{1em} 
\begin{algorithm}
\SetKwInOut{Input}{Input}\SetKwInOut{Output}{Output}
\SetKwInOut{Initialization}{Initialization}\SetKwInOut{Procedure}{Procedure}
\SetKwInOut{Return}{Return}

\Input{$y_{1}^{k}$, $p_{\mathbf{X}}$, $n$, $v$}
\Output{$g_{nv}\left( y_{1}^{k}\right)$}
\BlankLine
\Initialization{\\
$t_{0}\leftarrow m^{-1}\left( v\right) $\;
$g_{0}(\left. x_{1}\right\vert x_{1}^{0})\leftarrow \pi_{u}^{v}(x_{1})$\;
$u_{1,1}\leftarrow u(y_{1})$\;
}
\Procedure{\\
\For{$i\leftarrow 1$ \KwTo $k-1$}{
$m_{i} \leftarrow$ (\ref{def:m_i_chap2})\;
$t_{i} \leftarrow m^{-1}(m_{i})$ $*$\;
$\alpha \leftarrow$(\ref{def:alpha_chap2})\;
$\beta\leftarrow$(\ref{def:beta_chap2})\;
Calculate $C_{i}$\;
$g(\left. y_{i+1}\right\vert y_{1}^{i})\leftarrow$(\ref{def:g_i_chap2})\;
$u_{1,i+1}\leftarrow u_{1,i}+u(y_{i+1})$\;
}
Compute $g_{nv}\left( y_{1}^{k}\right) \leftarrow$(\ref{def:g_nv_chap2})\;
}
\Return{$g_{nv}(y_{1}^{k})$}
\caption{Evaluation of $g_{nv}(y_{1}^{k})$\label{algo_g_v_chap2}}
\end{algorithm}\DecMargin{1em}

\IncMargin{1em} 
\begin{algorithm}[Hhtbp]
\SetKwInOut{Input}{Input}\SetKwInOut{Output}{Output}
\SetKwInOut{Initialization}{Initialization}\SetKwInOut{Procedure}{Procedure}
\SetKwInOut{Return}{Return}

\Input{ $y_{1}^{n}$, $p_{\mathbf{X}}$, $n$, $k$, $a$, $M$}
\Output{$g_{nA}\left( y_{1}^{n}\right)$}
\BlankLine
\Procedure{\\
\For{$m\leftarrow 1$ \KwTo $M$}{
Simulate $v_{m}$ with density (\ref{dens_exp_s})\;
Calculate $g_{nv_{m}}\left( y_{1}^{k}\right) $ with Algorithm 1\;
Calculate $g_{nv_{m}}\left(y_{k+1}^{n}|y_{1}^{k}\right)\leftarrow(\ref{complementk,n})$\;
Calculate $g_{nv_{m}}\left( y_{1}^{n}\right)\leftarrow g_{nv_{m}}\left(
y_{1}^{k}\right) g_{nv_{m}}\left( y_{k+1}^{n}|y_{1}^{k}\right) $
}
Compute $g_{nA}\left(y_{1}^{n}\right)\leftarrow\frac{1}{M}\sum_{m=1}^{M}g_{nv_{m}}\left(y_{1}^{n}\right)$\;
}
\Return{$g_{nA}\left( y_{1}^{n}\right) $}
\caption{Evaluation of $g_{nA}\left( y_{1}^{n}\right) $\label{algo_g_A_chap2}}
\end{algorithm}\DecMargin{1em}

The calculation of $g_{nv}\left( y_{1}^{k}\right) $ above requires the value
of 
\begin{equation*}
C_{i}=\left( \int p_{\mathbf{X}}(x)\mathfrak{n}\left( \alpha \beta
+v,\beta ,u(x)\right) dx\right) ^{-1}.
\end{equation*}
This can be done through Monte Carlo simulation. The value of $M$ need not be very large.

\begin{remark}
Solving $t_{i}=m^{-1}(m_{i})$ might be difficult. It may happen that the
reciprocal function of $m$ is at hand, but even when $p_{\mathbf{X}}$ is the
Weibull density and $u(x)=x$, such is not the case. We can replace step $*$
by 
\begin{eqnarray}
t_{i+1}:=t_{i}-\frac{\left( m\left(t_{i}\right)+u_{i}\right) }{\left(
n-i\right)s^{2}\left(t_{i}\right)}.
\end{eqnarray}
Indeed since 
\begin{equation*}
m(t_{i+1})-m(t_{i})=-\frac{1}{n-i}\left( m(t_{i})+u_{i}\right)
\end{equation*}
with $u_{i}:=u\left(y_{i}\right)$, use a first order approximation to derive
that $t_{i+1}$ can be subtituted by $\tau _{i+1}$ defined through 
\begin{equation*}
\tau _{i+1}:=t_{i}-\frac{1}{\left( n-i\right) s^{2}(t_{i})}\left(
m(t_{i})+u_{i}\right).
\end{equation*}
In the moderate deviation scale the function $s^{2}(.)$ does not vary from $1
$ and the above approximation is fair. For the large deviation case, the
same argument applies, since $s^{2}(t_{i})$ keeps close to $s^{2}(t^{a}).$
\end{remark}

\IncMargin{1em} 
\begin{algorithm}[Hhtbp]
\SetKwInOut{Input}{Input}\SetKwInOut{Output}{Output}
\SetKwInOut{Initialization}{Initialization}\SetKwInOut{Procedure}{Procedure}
\SetKwInOut{Return}{Return}

\Input{ $p_{\mathbf{X}}$, $\delta$, $n$, $a$, $L$}
\Output{$k_{\delta}$}
\BlankLine
\Initialization{$k=1$}
\Procedure{\\
\While{$\delta\notin\overline{CI}(k)$}{
\For{$l\leftarrow 1$ \KwTo $L$}{
Simulate $Y_{1}^{k}(l)$ i.i.d. with density $p_{\mathbf{X}}$\;
$A\left( Y_{1}^{k}(l)\right) :=$(\ref{A(l)_chap2}) using Algorithm 2 \;
$B\left( Y_{1}^{k}(l)\right) :=$(\ref{B(l)_chap2}) using Algorithm 2 \;
}
Calculate $\overline{CI}(k)\leftarrow$(\ref{CIbarre_chap2})\;
$k:=k+1$\;
}
}
\Return{$k_{\delta}:=k$}
\caption{Calculation of $k_{\delta}$\label{algo_calcul_k_chap2}}
\end{algorithm}\DecMargin{1em}

\begin{center}
The next algorithms \ref{algo_simu_Y_chap2}, \ref{algo_simu_Y_1_k_chap2} and %
\ref{algo_P_n_chap2} provide the estimate of $P_{n}$.
\end{center}

The following algorithm provides a simple acceptance/rejection simulation
tool for $Y_{i+1}$ with density $g(\left. y_{i+1}\right\vert y_{1}^{i}).$

Denote $\mathfrak{N}$ the c.d.f. of a normal variate with parameter $\left(
\mu,\sigma^{2}\right) $ ,and $\mathfrak{N}^{-1}$ its inverse.

\IncMargin{1em} 
\begin{algorithm}[Hhtbp]
\SetKwInOut{Input}{Input}\SetKwInOut{Output}{Output}
\SetKwInOut{Initialization}{Initialization}\SetKwInOut{Procedure}{Procedure}
\SetKwInOut{Return}{Return}

\Input{$p$, $\mu$, $\sigma^{2}$}
\Output{$Y$}
\Initialization{\\
Select a density $f$ on $\left[ 0,1\right] $ and \newline
a positive constant $K$ such that \newline
$p\left( \mathfrak{N}^{-1}(x)\right) \leq Kf(x)$ for all $x$ in $\left[ 0,1%
\right] $
}
\Procedure{\\
\While{$Z<$ $p\left( \mathfrak{N}^{-1}(X)\right) $}{
Simulate $X$ with density $f$\;
Simulate $U$ uniform on $\left[ 0,1\right] $ independent of $X$\;
Compute $Z:=KUf(X)$\;
}
}
\Return{$Y:=\mathfrak{N}^{-1}(X)$}
\caption{Simulation of $Y$ with density proportional to $p(x)\mathfrak{n}\left( \mu,\sigma^{2},x\right) $\label{algo_simu_Y_chap2}}
\end{algorithm}

\IncMargin{1em} 
\begin{algorithm}[Hhtbp]
\SetKwInOut{Input}{Input}\SetKwInOut{Output}{Output}
\SetKwInOut{Initialization}{Initialization}\SetKwInOut{Procedure}{Procedure}
\SetKwInOut{Return}{Return}

\Input{$p_{\mathbf{X}}$, $\delta$, $n$, $v$}
\Output{$Y_{1}^{k}$}
\Initialization{\\
Set $k\leftarrow k_{\delta}$ with Algorithm 3\;
$t_{0}\leftarrow m^{-1}(v)$\;
}
\Procedure{\\
Simulate $Y_{1}$ with density $\pi_{u}^{v}$\;
$u_{1,1}\leftarrow u(Y_{1})$\;
\For{$i\leftarrow 1$ \KwTo $k-1$}{
$m_{i}\leftarrow$(\ref{def:m_i_chap2})\;
$t_{i}\leftarrow m^{-1}(m_{i})$\;
$\alpha\leftarrow$(\ref{def:alpha_chap2})\;
$\beta\leftarrow$(\ref{def:beta_chap2})\;
Simulate $Y_{i+1}$ with density $g(\left. y_{i+1}\right\vert y_{1}^{i})$ using Algorithm 4\;
$u_{1,i+1}\leftarrow u_{1,i}+u(Y_{i+1})$\;
}
}
\Return{$Y_{1}^{k}$}
\caption{Simulation of a sample $Y_{1}^{k}$ with density $g_{nv}$ \label{algo_simu_Y_1_k_chap2}}
\end{algorithm}

\begin{remark}
The paper \cite{BarbeBroniatowski1999} can be used in order to simulate $%
Y_{1}.$
\end{remark}

\IncMargin{1em} 
\begin{algorithm}[Hhtbp]
\SetKwData{Left}{left}\SetKwData{This}{this}\SetKwData{Up}{up}
\SetKwFunction{Union}{Union}\SetKwFunction{FindCompress}{FindCompress}
\SetKwInOut{Input}{Input}\SetKwInOut{Output}{Output}
\SetKwInOut{Initialization}{Initialization}\SetKwInOut{Procedure}{Procedure}
\SetKwInOut{Return}{Return}

\Input{ $p_{\mathbf{X}}$, $\delta$, $n$, $a$, $M$, $L$}
\Output{$\widehat{P_{n}}$}
\Initialization{\\
Set $k\rightarrow k_{\delta}$ with Algorithm 3\;
}
\Procedure{\\
\For{$l\leftarrow 1$ \KwTo $L$}{
Simulate $v_{l}$ with density (\ref{dens_exp_s})\;
Simulate $Y_{1}^{k}(l)$ with density $g_{nv_{l}}$ with Algorithm 5\;
Simulate $Y_{k+1}^{n}(l)$ i.i.d. with density $\pi_{u}^{\alpha_{l}}$\;
Calculate $g_{nA}\left( Y_{1}^{n}(l)\right) $ with Algorithm 2\;
Calculate $\widehat{P_{n}}(l)\leftarrow(\ref{Pchapl})$\;
}
Compute $\widehat{P_{n}}\leftarrow(\ref{Pchap})$\;
}
\Return{$\widehat{P_{n}}$}
\caption{Calculation of $\hat{P_{n}}$ \label{algo_P_n_chap2}}
\end{algorithm}

\begin{remark}
$\pi_{\mathbf{U}}^{\alpha_{l}}$ is defined as in (\ref{complementk,n}) 
\begin{equation*}
\alpha_{l}:=\frac{n}{n-k}\left(v_{l}-\frac{u_{1,k}}{n}\right)
\end{equation*}
as in (\ref{def:m_i_chap2}) and 
\begin{equation*}
u_{1,k}=\sum_{i=1}^{k}u(Y_{i}(l)).
\end{equation*}
\end{remark}

\section{Compared efficiencies of IS estimators}

\label{sec:improvedIS_chap2}

The situation which we face with our proposal lacks the possibility to
provide an order of magnitude of the variance our our IS estimate, since the
properties necessary to define it have been obtained only on \textit{typical
paths} under the sampling density $g_{nA}$ and not on the whole space $%
\mathbb{R}^{n}$ . This leads to a quasi-MSE measure for the performance of
the proposed estimator, which quantifies the variability evaluated on
classes of subsets of $\mathbb{R}^{n}$ whose probability goes to $1$ under
the sampling $g_{nA}.$ Not surprisingly the loss of performance with respect
to the optimal sampling density is due to the $n-k$ last i.i.d. simulations,
leading a quasi- MSE of the estimate proportional to $\sqrt{n-k}.$

\subsection{The efficiency of the classical IS scheme}

We first recall the definition of the classical IS sampling scheme and its
asymptotic performance. The r.v.'s $Y_{i}$'s in (\ref{Facteur_d_Importance})
are i.i.d. and have density $g=\pi_{u}^{a}$ , hence with $m(t)=a.$ See \cite%
{SadowskyBucklew1990} in the LDP case and \cite{Ermakov2007} in the MDP
case. The reason for this sampling scheme is the fact that in the large
deviation case, $a$ is the "dominating point" of the set $\left(
a,\infty\right) $ i.e. $a$ is such that the proxy of the conditional
distribution of $\mathbf{X}_{1}$ given $\left(\mathbf{U}_{1,n}>na\right)$ is 
$\Pi_{u}^{a}$; this is the basic form of the Gibbs conditioning principle.

Although developed for the large deviation case, the classical IS applies
for the moderate deviation case since for $a\rightarrow{E[u\left(\mathbf{X}%
\right)]}$ and $\left(a-E[u\left(\mathbf{X}\right)]\right)\sqrt{n}%
\rightarrow\infty$ it holds 
\begin{equation}
P\left( \left. \mathbf{X}_{1}\in B\right\vert \mathbf{U}_{1,n}>na\right)
=\left( 1+o(1)\right) \Pi_{u}^{a}(B)  \label{MDP}
\end{equation}
for any Borel set $B$ as $n\rightarrow\infty$. This follows as a consequence
of Sanov Theorem for moderate deviations (see \cite{Ermakov2007} and \cite%
{deAcosta1992}) and justifies the classical IS scheme in this range.

The classical IS is defined simulating $L$ times a random sample of $n$
i.i.d. r.v's $Y_{1}^{n}(l)$, $1\leq l\leq L,$ with tilted density $%
\pi_{u}^{a}$. The standard IS estimate is defined through%
\begin{equation*}
\overline{P_{n}}:=\frac{1}{L}\sum_{l=1}^{L}\mathds{1}_{\mathcal{E}%
_{n}}(Y_{1}^{n}(l))\frac{\prod_{i=1}^{n}p_{\mathbf{X}}(Y_{i}(l))}{%
\prod_{i=1}^{n}\pi_{u}^{a}(Y_{i}(l))}
\end{equation*}%
where the $X_{i}(l)$ are i.i.d. with density $\pi_{u}^{a}$ and $\mathds{1}_{%
\mathcal{E}_{n}}(Y_{1}^{n}(l))$ is as in (\ref{E_n}). Set 
\begin{equation*}
\overline{P_{n}}(l):=\mathds{1}_{\mathcal{E}_{n}}(Y_{1}^{n}(l))\frac{%
\prod_{i=1}^{n}p_{\mathbf{X}}(Y_{i}(l))}{\prod_{i=1}^{n}\pi_{u}^{a}(Y_{i}(l))%
}.
\end{equation*}%
The variance of $\overline{P_{n}}$ is given by 
\begin{equation*}
Var\overline{P_{n}}=\frac{1}{L}\left( E_{\Pi_{u}^{a}}\left( \overline{P_{n}}%
(l)\right) ^{2}-P_{n}^{2}\right) .
\end{equation*}%
The \textit{relative accuracy} of the estimate $\overline{P_{n}}$ is defined
through 
\begin{equation*}
RE(\overline{P_{n}}):=\frac{Var\overline{P_{n}}}{P_{n}^{2}}=\frac{1}{L}%
\left( \frac{E_{\Pi_{u}^{a}}\left( \overline{P_{n}}(l)\right) ^{2}}{P_{n}^{2}%
}-1\right) .
\end{equation*}%
The following result holds.

\begin{proposition}
\label{Prop rel efficiency standard IS}The relative accuracy of the estimate 
$\overline{P_{n}}$ is given by 
\begin{equation*}
RE(\overline{P_{n}})=\frac{\sqrt{2\pi}\sqrt{n}}{L}a(1+o(1)) 
\end{equation*}
as $n$ tends to infinity.
\end{proposition}

We will prove that no reduction of the variance of the estimator can be
achieved on subsets $B_{n}$ of $\mathbb{R}^{n}$ such that $%
\Pi_{u}^{a}(B_{n})\rightarrow{1}.$

The easy case when $\mathbf{U}_{1},...,\mathbf{U}_{n}$ are i.i.d. with
standard normal distribution and $u(x)=x$ is sufficient for our need.

The variance of the IS estimate of $P\left(\mathbf{U}_{1,n}>na\right)$ is
proportional to 
\begin{align*}
V & :=E_{P_{\mathbf{U}}}\mathds{1}_{\left(a,\infty\right) }\left(\frac{%
\mathbf{U}_{1,n}}{n}\right) \frac{p_{\mathbf{U}}\left( \mathbf{U}%
_{1}^{n}\right) }{\pi_{\mathbf{U}}^{a}\left( \mathbf{U}_{1}^{n}\right) }%
-P_{n}^{2} \\
& =E_{P_{\mathbf{U}}}\mathds{1}_{\left( a,\infty\right) }\left(\frac{\mathbf{%
U}_{1,n}}{n}\right) \left( \exp\frac{na^{2}}{2}\right) \left( \exp-a\mathbf{U%
}_{1,n}\right) -P_{n}^{2}
\end{align*}
A set $B_{n}$ resulting as reducing the MSE should penalize large values of $%
-\left( \mathbf{U}_{1}+...+\mathbf{U}_{n}\right) $ while bearing nearly all
the realizations of $\mathbf{U}_{1}+...+\mathbf{U}_{n}$ under the i.i.d.
sampling scheme $\pi_{\mathbf{U}}^{a}$ as $n$ tends to infinity. It should
therefore be of the form $\left( b,\infty\right) $ for some $b=b_{n}$ so that

(a) 
\begin{equation*}
\lim_{n\rightarrow\infty}E_{\Pi_{\mathbf{U}}^{a}}\mathds{1}_{\left(
b,\infty\right) }\left( \frac{\mathbf{U}_{1,n}}{n}\right) =1
\end{equation*}
and

(b)%
\begin{equation*}
\lim_{n\rightarrow\infty}\sup\frac{E_{P_{\mathbf{U}}}\mathds{1}_{\left(
a,\infty\right) \cap\left( b,\infty\right) }\left( \frac{\mathbf{U}_{1,n}}{n}%
\right) \frac{p_{\mathbf{U}}\left( \mathbf{U}_{1}^{n}\right) }{\pi_{\mathbf{U%
}}^{a}\left( \mathbf{U}_{1}^{n}\right) }}{V}<1
\end{equation*}
which means that the IS\ sampling density $\pi_{\mathbf{U}}^{a}$ can lead a
MSE defined by 
\begin{equation*}
MSE(B_{n}):=E_{P_{\mathbf{U}}}\mathds{1}_{\left( na,\infty\right) \cap\left(
nb,\infty\right) }\frac{p_{\mathbf{U}}\left( \mathbf{U}_{1}^{n}\right) }{%
\pi_{\mathbf{U}}^{a}\left( \mathbf{U}_{1}^{n}\right) }-P_{n}^{2}
\end{equation*}
with a clear gain over the variance indicator. However when $b\leq a$, (b)
does not hold and, when $b>a$, (a) does not hold.

So no reduction of this variance can be obtained by taking into account the
properties of the \textit{typical paths }generated under the sampling
density: a reduction of the variance is possible only by conditioning on
"small" subsets of the sample paths space. On no classes of subsets of $%
\mathbb{R}^{n}$ with probability going to $1$ under the sampling is it
possible to reduce the variability of the estimate, whose rate is definitely
proportional to $\sqrt{n},$ imposing a burden of order $L\sqrt{n}\alpha$ in
order to achieve a relative efficiency of $\alpha\%$ with respect to $P_{n}.$

\subsection{Efficiency of the adaptive twisted scheme}

We first put forwards a Lemma which assesses that large sets under the
sampling distribution $g_{nA}$ bear all what is needed to achieve a dramatic
improvement of the relative efficiency of the IS procedure. Its proof is
deferred to the Appendix.

\begin{lemma}
\label{Lemma set C_n for efficiency} Assume $k/n\rightarrow1.$ It then holds,

\begin{enumerate}
\item There exist sets $C_{n}$ in $\mathbb{R}^{n}$ such that

\begin{itemize}
\item $\lim_{n\rightarrow\infty}G_{nA}\left( C_{n}\right) =1$

\item for any $y_{1}^{n}$ in $C_{n}$, $\vert \frac{p_{nA}}{g_{nA}}%
\left(y_{1}^{k}\right)-1\vert<\delta_{n}$ with $\delta_{n}$ as in (\ref%
{vitesse_chap2_a_fixed}).

\item 
\item when $a\rightarrow E\mathbf{U}$ (moderate deviation), 
\begin{equation}
t^{k}s(t^{k})=a\left( 1+o(1)\right)   \label{order of t_k}
\end{equation}

\item when $\lim_{n\rightarrow \infty }a_{n}$  is larger than $E\mathbf{U}$
(large deviation) , $t^{k}s(t^{k})$ remains bounded away from $0$ and
infinity.
\end{itemize}
\end{enumerate}
\end{lemma}

\bigskip 

We now evaluate the Mean Square Error of the adaptive twisted IS algorithm
on this family of sets. Let

\begin{equation*}
RE\left( \widehat{P_{n}}\right) =\frac{1}{L}\left( \frac{E_{G_{nA}}\left( %
\mathds{1}_{C_{n}}\widehat{P_{n}}(l)\right) ^{2}}{P_{n}^{2}}-1\right) .
\end{equation*}

We prove that

\begin{proposition}
\label{Prop:efficiency_class_IS} The relative accuracy of the estimate $\hat{%
P_{n}}$ is given by 
\begin{equation*}
RE(\widehat{P_{n}})=\frac{\sqrt{2\pi}\sqrt{n-k-1}}{L}a(1+o(1))
\end{equation*}
as $n$ tends to infinity.
\end{proposition}

\begin{proof}
Using the definition of $C_{n}$ we get 
\begin{equation*}
E_{G_{nA}}\left( \mathds{1}_{C_{n}}\widehat{P_{n}}(l)\right) ^{2} 
\end{equation*}
\begin{equation*}
=P_{n}E_{P_{nA}}\mathds{1}_{C_{n}}(Y_{1}^{n})\frac{p_{\mathbf{X}%
}(Y_{1}^{k})p_{\mathbf{X}}(Y_{k+1}^{n})}{g_{nA}(Y_{1}^{k})g_{nA}(%
\left.Y_{k+1}^{n}\right\vert Y_{1}^{k})}
\end{equation*}
\begin{equation*}
\leq P_{n}(1+\delta_{n})E_{P_{nA}}\mathds{1}_{C_{n}}(Y_{1}^{n})\frac{p_{%
\mathbf{X}}(Y_{1}^{k})}{p(\left. Y_{1}^{k}\right\vert \mathcal{E}_{n})}\frac{%
p_{\mathbf{X}}(Y_{k+1}^{n})}{g_{nA}(\left. Y_{k+1}^{n}\right\vert Y_{1}^{k})}
\end{equation*}
\begin{equation*}
=P_{n}^{2}(1+\delta_{n})E_{P_{nA}}\mathds{1}_{C_{n}}(Y_{1}^{n})\frac{1}{%
p(\left. \mathcal{E}_{n}\right\vert Y_{1}^{k})}\frac{p_{\mathbf{X}%
}(Y_{k+1}^{n})}{g_{nA}(\left. Y_{k+1}^{n}\right\vert Y_{1}^{k})}
\end{equation*}
\begin{equation*}
=P_{n}^{2}(1+\delta_{n})\sqrt{2\pi}\sqrt{n-k-1}
\end{equation*}
\begin{equation*}
E_{P_{nA}}\mathds{1}_{C_{n}}(Y_{1}^{n})t^{k}s(t^{k})(1+o(1))
\end{equation*}
\begin{equation*}
=P_{n}^{2}a\sqrt{2\pi}\sqrt{n-k-1}(1+o(1)).
\end{equation*}
The third line is Bayes formula. The fourth line is Lemma \ref%
{Lemma_Jensen_chap2} (see the Appendix). The fifth line uses (\ref{order of
t_k}) and uniformity in Lemma \ref{Lemma_Jensen_chap2}, where the conditions
in Corollary 6.1.4 of \cite{Jensen1995} are easily checked since, in his
notation, $J(\theta )=\mathbb{R}$ , condition (i) holds for $\theta$ in a
neighborhood of $0$ ($\Theta_{0}$ indeed is restricted to such a set in our
case), (ii) clearly holds and (iii) is a consequence of the assumption on
the characteristic function of $u\left( \mathbf{X}_{1}\right).$
\end{proof}

\bigskip

\section{Simulation results}

\label{sec:simu_chap2}

\subsection{The gaussian case}

The random variables $X_{i}^{\prime}s$ are i.i.d. with normal distribution
with mean $0$ and variance $1.$ The case treated here is $P_{n}=P\left( 
\frac{\mathbf{S}_{1,n}}{n}>a\right)=0.009972$ with $n=100,$ and $a=0.232.$
We build the curve of the estimate of $P_{n}$ (solid lines) and the two
sigma confidence interval (dot lines) with respect to $k$. The value of $L$
is $L=2000.$

\begin{figure}[!ht]
\centering \includegraphics*[ scale=0.4]{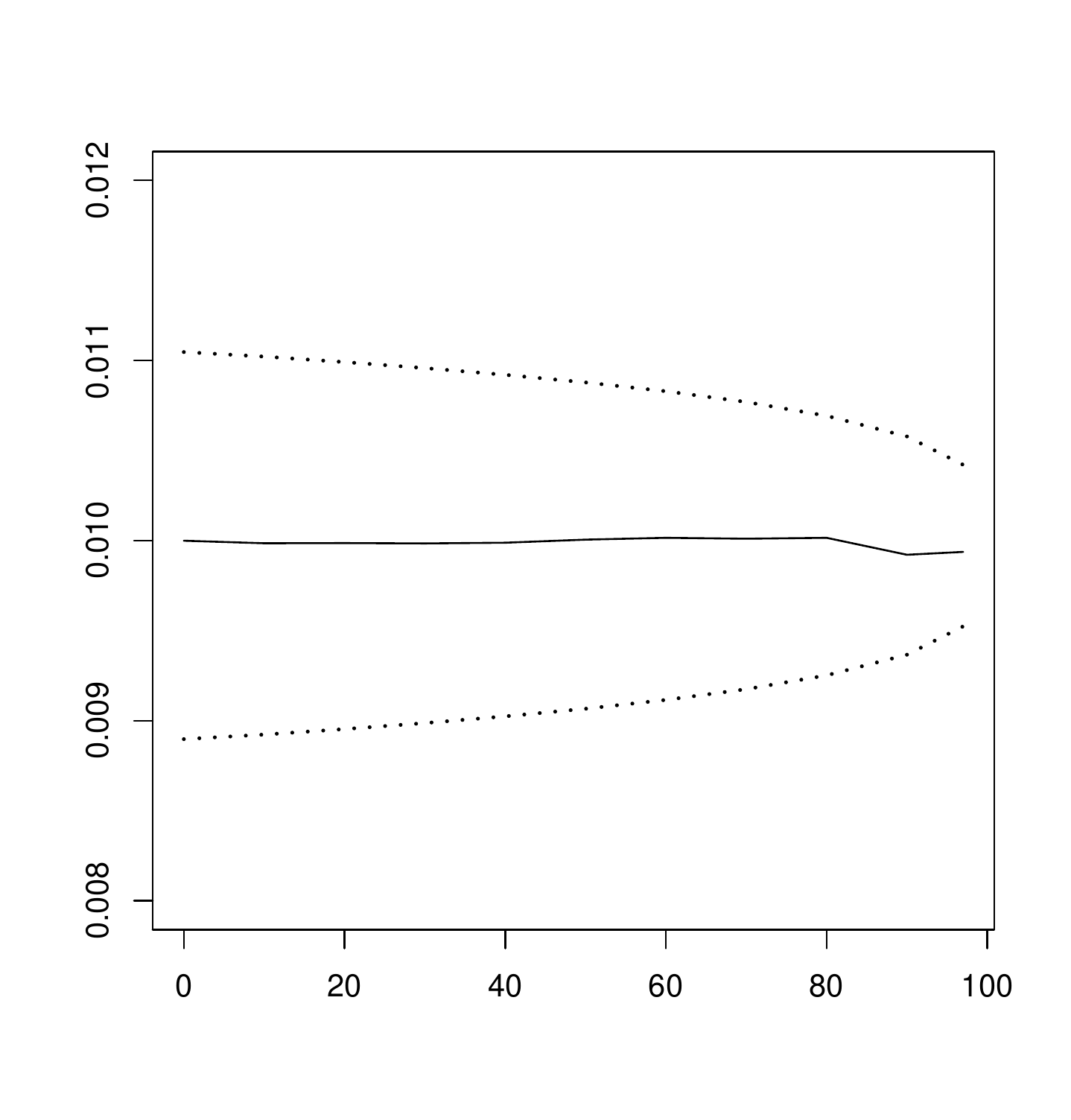}
\caption{Curve of $\widehat{P_{n}}$ (solid line) in the normal case along
with the two sigma confidence interval (dotted lines) as function of $k$
with $n=100$ for $L=2000$ instances.}
\end{figure}

\subsection{The exponential case}

The random variables $X_{i}^{\prime}s$ are i.i.d. with exponential
distribution with parameter $1$ on $\left( -1,\infty\right) .$ The case
treated here is $P_{n}=P\left( \frac{\mathbf{S}_{1,n}}{n}>a\right)=0.013887$
with $n=100,$ and $a=0.232.$ The solid lines is the estimate of $P_{n}$, the
dot lines are the two sigma confidence interval. Abscissa is $k.$

\begin{figure}[!ht]
\centering \includegraphics*[ scale=0.4]{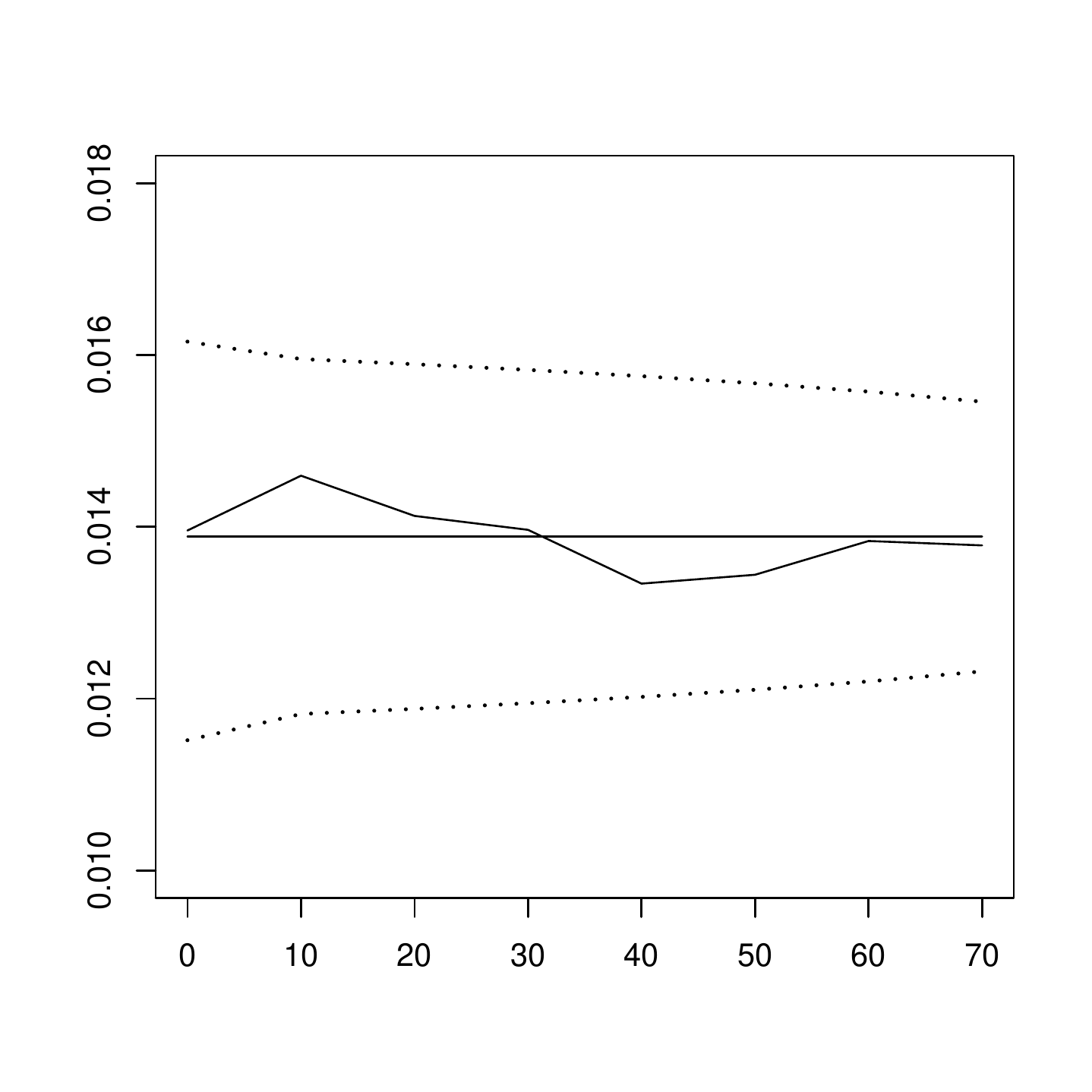}
\caption{Curve of $\widehat{P_{n}}$ (solid line) in the exponential case
along with the two sigma confidence interval (dotted lines) as function of $k
$ with $n=100$ for $L=2000$ instances.}
\end{figure}

Figure \ref{Figure_empirical_MSE_vs_empirical_MSE_chap2} shows the
ratio of the empirical value of the MSE of the adaptive estimate w.r.t. the
empirical MSE of the i.i.d. twisted one, in the exponential case with $%
P_{n}=10^{-2}$ and $n=100.$ The value of $k$ is growing from $k=0$ (i.i.d.
twisted sample) to $k=70$ (according to the rule of section \ref%
{sec:howfar_chap2}). This ratio stabilizes to $\sqrt{n-k}/\sqrt{n}$ for $%
L=2000$. The abscissa is $k$ and the solid line is $k\rightarrow\sqrt {n-k}/%
\sqrt{n}.$

\begin{figure}[!ht]
\centering \includegraphics*[ scale=0.5]{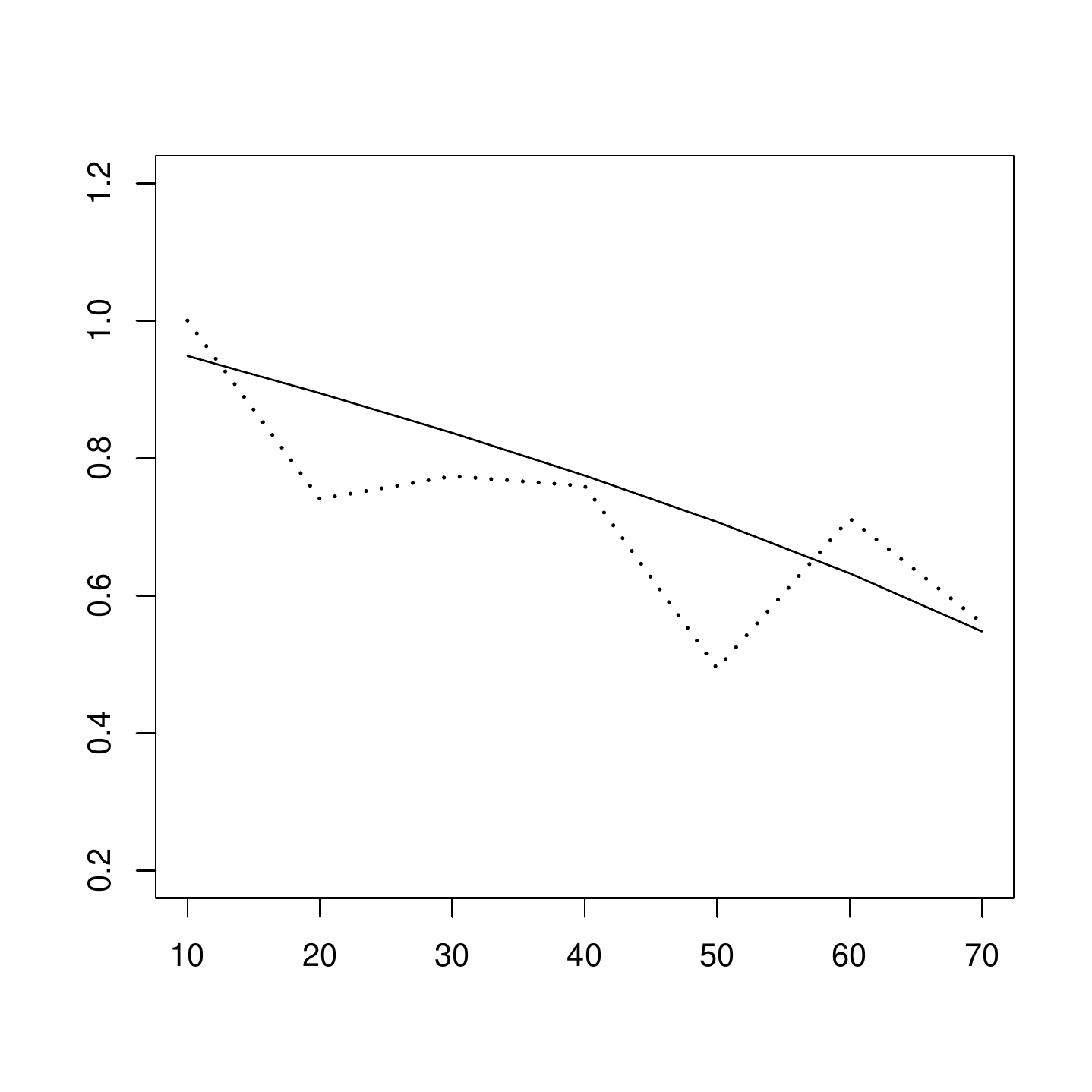}
\caption{Ratio of the empirical value of the MSE of the adaptive estimate
w.r.t. the empirical MSE of the i.i.d. twisted one (dotted line) along with
the true value of this ratio (solid line) as a function of $k.$}
\label{Figure_empirical_MSE_vs_empirical_MSE_chap2}
\end{figure}

\subsection{A comparison study with the classical twisted IS scheme}

This section compares the performance of the present approach with respect
to the standard tilted one as described in Section \ref{sec:intro_chap2}.

Consider a random sample $X_{1},...,X_{100}$ where $X_{1}$ has a normal
distribution $N(0.05,1)$\ and let 
\begin{equation*}
\mathcal{E}_{100}:=\left\{ x_{1}^{100}:\frac{\left\vert
x_{1}+...+x_{100}\right\vert }{100}>0.28\right\}
\end{equation*}
for which 
\begin{equation*}
P_{100}=P\left( \left( X_{1},...,X_{100}\right) \in\mathcal{E}_{100}\right)
=0.01120.
\end{equation*}
Our interest is to show that in this simple dissymetric case a direct
extension of our proposal provides a good estimate, while the standard IS
scheme ignores a part of the event $\mathcal{E}_{100}.$ The standard i.i.d.
IS scheme introduces the dominating point $a=0.28$ and the family of i.i.d.
tilted r.v's with common $N(a,1)$ distribution. The resulting estimator of $%
P_{100}$ is $0,01074$ (with $L=1000)$, indicating that the event $%
S_{1,100}/100<-0.28$ is ignored in the evaluation of $P_{100}$, inducing a
bias in the estimation. Since the simulated r.v's are independent under the
tilted distribution the Importance factor oscillates wildly. Also the hit
rate is of order 50\%. It can also be seen that $S_{1}^{100}/100<-0.28$ is
never visited through the procedure.

This example is not as artificial as it may seem; indeed it leads to a two
dominating points situation which is quite often met in real life. Exploring
at random the set of interest under the distribution of $\left(
x_{1}+...+x_{100}\right) /100$ under $\mathcal{E}_{100}$ avoids any search
for dominating points. A further paper in $\mathbb{R}^{d}$ explores the
advantage of this method, which already proves to compare favorably with
usual methods on $\mathbb{R}$.

Drawing $L$ i.i.d. points $v_{1},...,v_{L}$ according to the distribution of 
$S_{1,100}/100$ conditionally upon $\left\vert S_{1,100}\right\vert /100>0.28
$ we evaluate $P_{100}$ with $k=99$; note that in the gaussian case Theorem %
\ref{thm:egal_chap2} provides an exact description of the conditional
density of $X_{1}^{k}$ for all $k$ between $1$ and $n$, and therefore the
same nearly holds in Theorem \ref{Thm:approx_largeSets_chap2_a_fixed}.
Simulating the $v_{i}$'s in this toy case is easy; just simulate samples $%
X_{1},...,X_{100}$ under $N(0.05,1)$ until $\mathcal{E}_{100}$ is reached.
The resulting value of the estimate is $0.01125$ which is fairly close to $%
P_{100}.$

As expected the Importance factor is very close to $P_{100}$ for all sample
paths $X_{1}^{n}$ simulated under $G_{nA}$; this is in accordance with
Theorem \ref{thm:egal_chap2}. Also the hit rate is very close to 100\%.

The histograms pertaining to the Importance factor are as follows (Figures
12 and 13).

\begin{figure}[!ht]
\centering \includegraphics*[ scale=0.5]{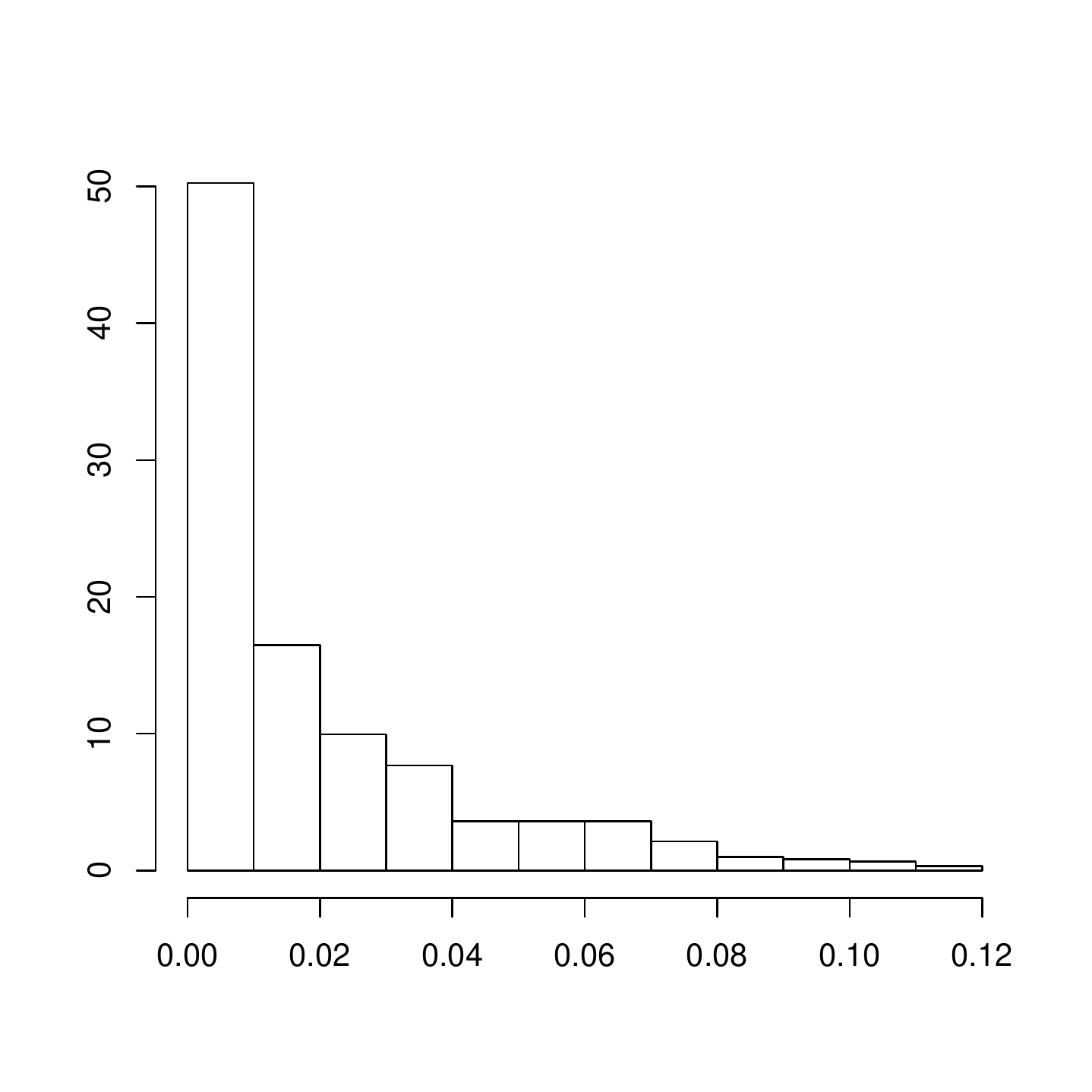}
\caption{Histogram of Importance Factor with $k=1$ and $n=100$ for $L=1000$
instances.}
\end{figure}

\begin{figure}[!ht]
\centering \includegraphics*[ scale=0.5]{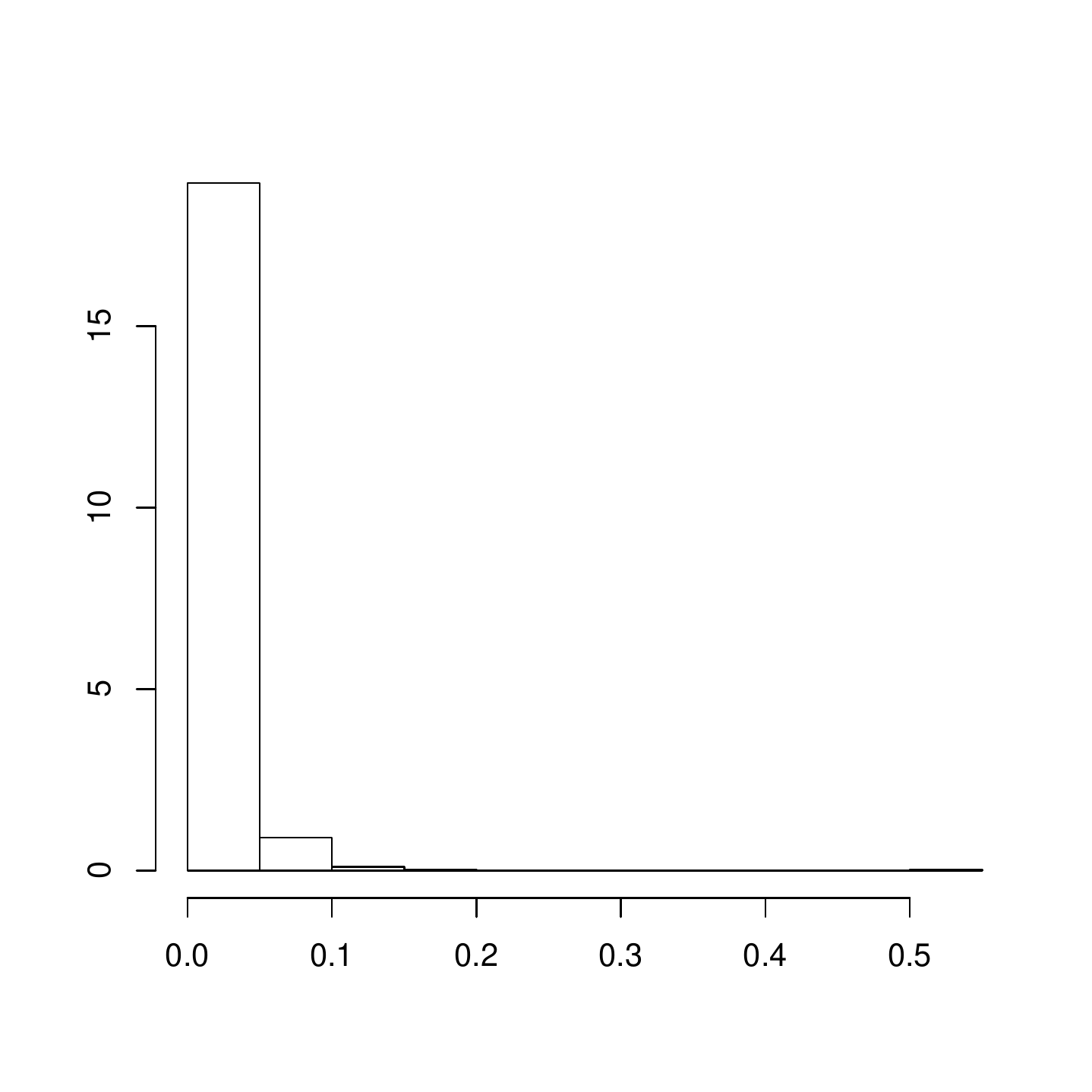}
\caption{Histogram of Importance Factor with $k=99$ and $n=100$ for $L=1000$
instances.}
\end{figure}

It is also interesting to draw the hit rate as a function of $k.$ When $k=1$
then this rate is close to 50\%, since the present algorithms coincides with
the classical i.i.d. IS scheme. As $k$ increases, the hit rate approaches
100\%; the value of $L$ is 1000.

\begin{figure}[!ht]
\centering \includegraphics*[ scale=0.5]{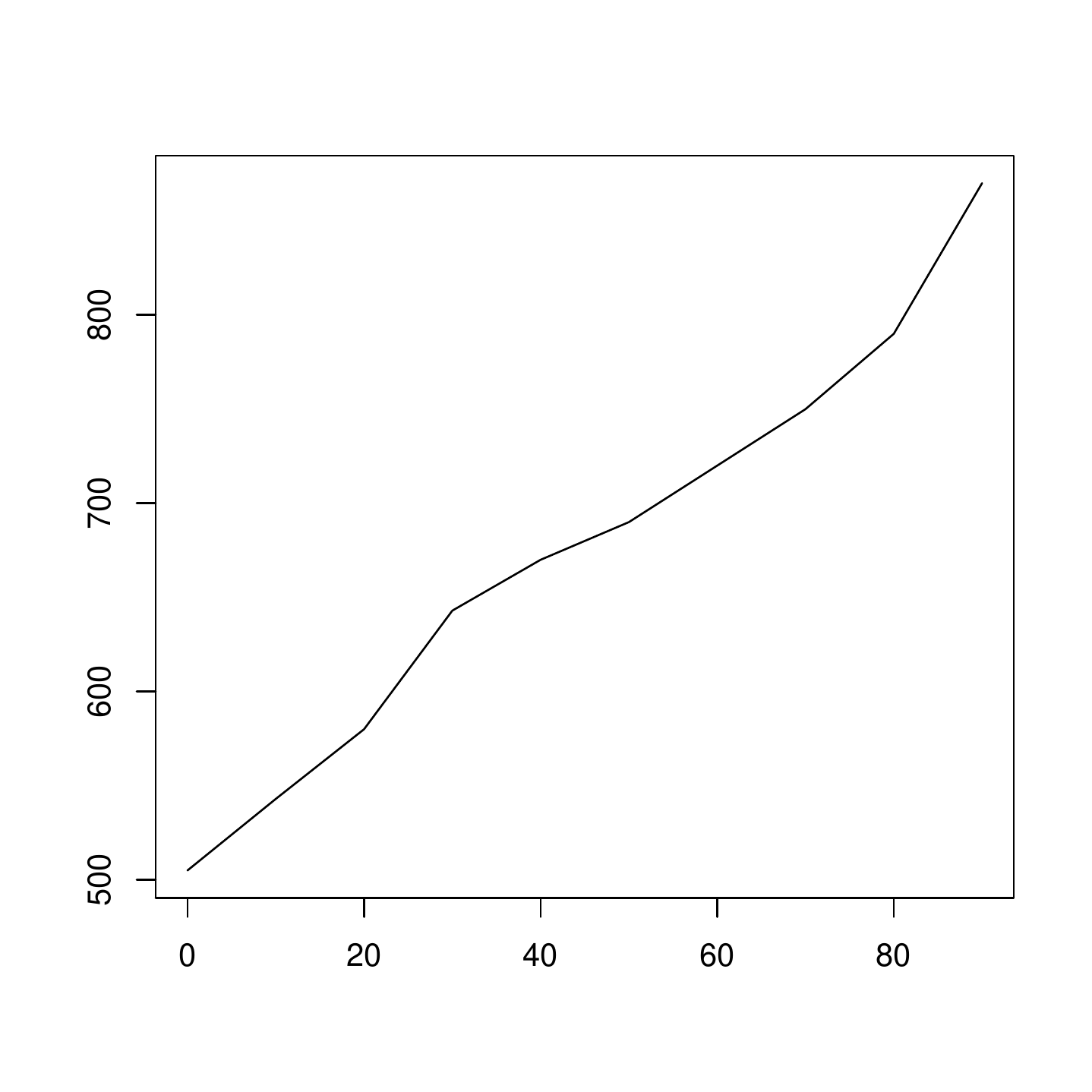}
\caption{Curve of the hit Rate as a function of $k$ with $n=100$ for $L=1000$
instances.}
\end{figure}

\section{Appendix}

\label{sec:appendix}

The following lemma provide asymptotic formula for the tail probability of $%
\mathbf{U}_{1,n}$ under the hypothesis and notations of section 3. Define

\begin{equation*}
I_{\mathbf{U}}(x):=xm^{-1}\left( x\right) -\log \phi _{\mathbf{U}}\left(
m^{-1}\left( x\right) \right)
\end{equation*}

\begin{lemma}
(see \cite{Jensen1995}, Corollary 6.4.1) \label{Lemma_Jensen_chap2} Under
the same hypotheses and notations as section 3, 
\begin{equation*}
P\left( \frac{\mathbf{U}_{1,n}}{n}>a\right) =\frac{\exp -nI_{\mathbf{U}}(a)}{%
\sqrt{2\pi }\sqrt{n}\psi (a)}\left( 1+O(\frac{1}{\sqrt{n}})\right)
\end{equation*}%
where $\psi(a):=m^{-1}(a)s(m^{-1}(a)).$
\end{lemma}

\subsection{Proof of Theorem \protect\ref{Thm:approx_largeSets_chap2_a_fixed}%
}

\subsubsection{Two Lemmas pertaining to the partial sum under its final value%
}

\begin{lemma}
\label{LemmaMomentsCond_chap2} Suppose that (V) holds. Then (i)$E_{P_{nA}}%
\mathbf{U}_{1}=a+o(1),$ (ii) $E_{P_{nA}}\mathbf{U}_{1}^{2}=a^{2}+s^{2}\left(
m^{-1}(a)\right) +o(1)$ and (iii) $E_{P_{nA}}\mathbf{U}_{1}\mathbf{U}%
_{2}=a^{2}+o(1).$
\end{lemma}

\begin{proof}
We make use of Lemma 23 of \cite{BroniatowskiCaron2011}, meaning $E_{P_{nv}}[\mathbf{U}_{1}]=v.$ It holds 
\begin{equation*}
E_{P_{nA}}\mathbf{U}_{1}=\int_{a}^{\infty }\left( E_{P_{nv}}\mathbf{U}%
_{1}\right) p\left( \left. \mathbf{U}_{1,n}/n=v\right\vert \mathbf{U}%
_{1,n}>na\right) dv.
\end{equation*}%
Integration by parts yields, 
\begin{equation*}
E_{P_{nA}}\mathbf{U}_{1}=a+\int_{a}^{\infty }P\left( \left. \mathbf{U}%
_{1,n}/n>v\right\vert \mathbf{U}_{1,n}>na\right) dv.
\end{equation*}%
Using Lemma \ref{Lemma_Jensen_chap2} and Chernoff inequality, 
\begin{equation*}
\int_{a}^{\infty }P\left( \left. \mathbf{U}_{1,n}/n>v\right\vert \mathbf{U}%
_{1,n}>na\right) dv\leq {\sqrt{2\pi }\psi (a)\sqrt{n}\int_{a}^{\infty }\exp
[n\left( I_{\mathbf{U}}(a)-I_{\mathbf{U}}(v)\right) ]}
\end{equation*}%
where $\psi (a)$ is defined in Lemma \ref{Lemma_Jensen_chap2}.

Finally, using $I_{\mathbf{U}}(v)>I_{\mathbf{U}}^{^{\prime }}(a)v+I_{\mathbf{%
U}}(a)-aI_{\mathbf{U}}^{^{\prime }}(a),$ and integrating 
\begin{equation*}
\int_{a}^{\infty }P\left( \left. \mathbf{U}_{1,n}/n>v\right\vert \mathbf{U}%
_{1,n}>na\right) dv\leq {\frac{\sqrt{2\pi }s(m^{-1}(a))}{\sqrt{n}}}.
\end{equation*}%
Hence, $E_{P_{nA}}\mathbf{U}_{1}=a+o(1).$

Insert $E_{P_{nv}}\mathbf{U}_{1}^{2}=v^{2}+s^{2}\left( m^{-1}(a)\right)
+O\left( \frac{1}{n}\right) $ in 
\begin{equation*}
E_{p_{nA}}\mathbf{U}_{1}^{2}=\int_{a}^{\infty }E_{P_{nv}}\mathbf{U}%
_{1}^{2}p\left( \left. \mathbf{U}_{1,n}/n=v\right\vert \mathbf{U}%
_{1,n}>na\right) dv.
\end{equation*}%
Firstly, by integration by parts, Lemma \ref{Lemma_Jensen_chap2} and
Chernoff inequality, 
\begin{equation*}
\int_{a}^{\infty }v^{2}p\left( \left. \mathbf{U}_{1,n}/n=v\right\vert 
\mathbf{U}_{1,n}>na\right) dv=a^{2}+o(1)
\end{equation*}%
Indeed, since (C) implies $nm^{-1}(a)\rightarrow {\infty }$ when n tends to $%
\infty $, it holds 
\begin{equation*}
\int_{a}^{\infty }vp\left( \left. \mathbf{U}_{1,n}/n>v\right\vert \mathbf{U}%
_{1,n}>na\right) dv\leq {\frac{s(m^{-1}(a))}{\sqrt{n}}\left( a+\frac{1}{%
nm^{-1}(a)}\right) }.
\end{equation*}

Secondly, 
\begin{equation*}
\int_{a}^{\infty }V(v)p\left( \left. \mathbf{U}_{1,n}/n=v\right\vert \mathbf{%
U}_{1,n}>na\right) dv=
\end{equation*}%
\begin{equation*}
s^{2}(m^{-1}(a))+2\int_{a}^{\infty }V^{^{\prime
}}(v)P\left( \left. \mathbf{U}_{1,n}/n>v\right\vert \mathbf{U}%
_{1,n}>na\right) dv.
\end{equation*}%
Using Lemma \ref{Lemma_Jensen_chap2}, Chernoff inequality and $I_{\mathbf{U}%
}(v)>I_{\mathbf{U}}^{^{\prime }}(a)v+I_{\mathbf{U}}(a)-aI_{\mathbf{U}%
}^{^{\prime }}(a),$ it holds under condition (V), 
\begin{eqnarray*}
&&\int_{a}^{\infty }V^{^{\prime }}(v)P\left( \left. \mathbf{U}%
_{1,n}/n>v\right\vert \mathbf{U}_{1,n}>na\right) dv \\
&\leq &s(m^{-1}(a))\left( \sqrt{n}{m^{-1}(a)\int_{a}^{\infty }V^{^{\prime
}}(v)\exp \left( -nm^{-1}(a)(v-a)\right) dv}\right) 
\end{eqnarray*}%
and 
\begin{equation*}
\int_{a}^{\infty }V(v)p\left( \left. \mathbf{U}_{1,n}/n=v\right\vert \mathbf{%
U}_{1,n}>na\right) dv=s^{2}(m^{-1}(a))+o(1).
\end{equation*}

The third term is handled similarly due to the fact that the $O(1/n)$
consists in a sum of powers of $v$.

For $E_{P_{nA}}\mathbf{U}_{1}\mathbf{U}_{2}=a^{2}+o(1)$, the proof is
similar.
\end{proof}

Lemma \ref{LemmaMomentsCond_chap2} yields the maximal inequality stated in
Lemma 22 of \cite{BroniatowskiCaron2011} under the condition $\left(\mathbf{U}_{1,n}>na\right).$ We also need the order of magnitude of the maximum of $\left( \left\vert \mathbf{U}_{1}\right\vert ,...,\left\vert \mathbf{U}_{k}\right\vert \right) $ under $P_{nA}$ which is stated in the following result.

\begin{lemma}
\label{Lemma_max_U_i_under_E_n_chap2} It holds for all $k$ between $1$ and $n
$ 
\begin{equation*}
\max \left( \left\vert \mathbf{U}_{1}\right\vert ,...,\left\vert \mathbf{U}%
_{k}\right\vert \right) =O_{P_{nA}}(\log n).
\end{equation*}
\end{lemma}

\begin{proof}
Using the same argument as in Lemma 23 of \cite{BroniatowskiCaron2011}, we consider the case when the r.v's $\mathbf{U}_{i}$ take non negative values. We prove that 
\begin{equation*}
\lim_{n\rightarrow \infty }P_{nA}\left( \max \left( \mathbf{U}_{1},...,%
\mathbf{U}_{k}\right) >t_{n}\right) =0
\end{equation*}%
when 
\begin{equation*}
\lim_{n\rightarrow \infty }\frac{t_{n}}{\log n}=\infty .
\end{equation*}%
It holds 
\begin{eqnarray*}
P_{nA}\left( \max \left( \mathbf{U}_{1},...,\mathbf{U}_{k}\right)
>t_{n}\right)  &=&\int_{a}^{a+c}P_{nv}\left( \left. \max \left( \mathbf{U}%
_{1},...,\mathbf{U}_{k}\right) >t_{n}\right\vert \mathbf{U}_{1,n}/n=v\right) 
\\
&&p\left( \left. \mathbf{U}_{1,n}/n=v\right\vert \mathbf{U}_{1,n}>na\right)
dv \\
&&+\int_{a+c}^{\infty }P_{nv}\left( \left. \max \left( \mathbf{U}_{1},...,%
\mathbf{U}_{k}\right) >t_{n}\right\vert \mathbf{U}_{1,n}/n=v\right)  \\
&&p\left( \left. \mathbf{U}_{1,n}/n=v\right\vert \mathbf{U}_{1,n}>na\right)
dv \\
&=&:I+II.
\end{eqnarray*}%
Now, using the same arguments as before, 
\begin{equation*}
II\leq \frac{P\left( \mathbf{U}_{1,n}/n>a+c\right) }{P\left( \mathbf{U}%
_{1,n}/n>a\right) } \\
\leq {\frac{m^{-1}(a)s(m^{-1}(a))}{m^{-1}(a+c)s(m^{-1}(a+c))}\exp \left(
-ncm^{-1}(a)\right) }
\end{equation*}
Since $c$ is fixed and $m^{-1}(a)$ is bounded , $II\rightarrow {0}$ under
(C).

Furthermore by Lemma 23 of \cite{BroniatowskiCaron2011}, 
\begin{eqnarray*}
\lim_{n\rightarrow \infty }P\left( \left. \max \left( \mathbf{U}_{1},...,\mathbf{U}_{n}\right) >t_{n}\right\vert \mathbf{U}_{1,n}/n=v\right)=:\lim_{n\rightarrow \infty }r_{n}=0
\end{eqnarray*}
when $v\in \left( a,a+c\right) $ .
Hence 
\begin{equation*}
I\leq r_{n}(1+o(1))\rightarrow 0.
\end{equation*}%
This proves the Lemma.
\end{proof}

We now prove Theorem \ref{Thm:approx_largeSets_chap2_a_fixed}(i).

\textit{Step 1}.We first prove that the integral (\ref{etoile_chap2}) can be
reduced to its principal part, namely that 
\begin{equation*}
p_{nA}(Y_{1}^{k})=\left( 1+o_{P_{nA}}\left( 1\right) \right)
\end{equation*}%
\begin{equation}
\int_{a}^{a+c}p(\left. \mathbf{X}_{1}^{k}=Y_{1}^{k}\right\vert \mathbf{U}%
_{1,n}/n=v)p(\left. \mathbf{U}_{1,n}/n=v\right\vert \mathbf{U}_{1,n}>na)dv
\label{Reductiona+c}
\end{equation}%
holds for any fixed $c>0.$

Apply Bayes formula to obtain%
\begin{equation*}
p_{nA}(Y_{1}^{k})=\frac{np_{\mathbf{X}}\left( Y_{1}^{k}\right) }{\left(
n-k\right) }
\end{equation*}%
\begin{equation*}
\frac{\int_{a}^{\infty }p\left( \frac{\mathbf{U}_{k+1,n}}{n-k}=\frac{n}{n-k}%
\left( t-\frac{k\overline{U_{1,k}}}{n}\right) \right) dt}{P\left( \mathbf{U}%
_{1,n}>na\right) }
\end{equation*}%
where $\overline{U_{1,k}}:=\frac{U_{1,k}}{k}.$

Denote%
\begin{equation*}
I:=\frac{P\left( \frac{\mathbf{U}_{k+1,n}}{n-k}>m_{k}+\frac{nc}{n-k}\right)}{%
P\left( \frac{\mathbf{U}_{k+1,n}}{n-k}>m_{k}\right)}.
\end{equation*}
with 
\begin{equation*}
m_{k}=\frac{n}{n-k}\left(a-\frac{k\overline{U_{1,k}}}{n}\right).
\end{equation*}
Then (\ref{Reductiona+c}) holds whenever $I\rightarrow 0$ (under $P_{nA}$).

Under $P_{nA}$ it holds 
\begin{equation*}
\overline{U_{1,n}}=a+O_{P_{nA}}\left( \frac{1}{nm^{-1}(a)}\right) .
\end{equation*}%
A similar result as Lemma 22 holds under condition $\left( 
\mathbf{U}_{1,n}>na\right) $, using Lemma 21; namely it
holds 
\begin{equation*}
\max_{0\leq i\leq k-1}\left\vert \overline{U_{i+1,n}}\right\vert
=a+o_{P_{nA}}\left( \epsilon _{n}\right) .
\end{equation*}%
Using both results, it holds 
\begin{equation}
m_{k}=a+O_{P_{nA}}\left( v_{n}\right)  \label{inter1}
\end{equation}
with $v_{n}=\max\left(\epsilon_{n},\frac{1}{(n-k)m^{-1}(a)}\right)$ which
tends to $0$ under (C).

We now prove that $I\rightarrow 0.$ Using once more Lemma \ref%
{Lemma_Jensen_chap2} yields 
\begin{equation*}
I\leq {\frac{m^{-1}(m_{k})s(m^{-1}(m_{k}))}{m^{-1}(m_{k}+\frac{nc}{n-k}%
)s(m^{-1}(m_{k})+\frac{nc}{n-k})}}
\end{equation*}%
\begin{equation*}
\exp \left( -(n-k)\left( I_{\mathbf{U}}\left( m_{k}+\frac{nc}{n-k}\right)
-I_{\mathbf{U}}\left( m_{k}\right) \right) \right) .
\end{equation*}

Now by convexity of the function $I_{\mathbf{U}},$ and (\ref{inter1}), 
\begin{eqnarray*}
&&\exp -\left( n-k\right) \left( I_{\mathbf{U}}\left( m_{k}+\frac{nc}{n-k}%
\right) -I_{\mathbf{U}}\left( m_{k}\right) \right) \\
&\leq &\exp -ncm^{-1}(m_{k})=\exp -nc\left[ m^{-1}(a)+\frac{1}{V(a+\theta
O_{P_{nA}}(v_{n}))}O_{P_{nA}}(v_{n})\right]
\end{eqnarray*}%

for some $\theta $ in $\left( 0,1\right) .$ which tends to $0$ under $P_{nA}$
when (A) and (C) hold. By monotonicity of $t\rightarrow m(t)$ and
condition (C) the ratio in $I$ is bounded.

We have proved that 
\begin{equation*}
I=O_{P_{nA}}\left( \exp -ncm^{-1}(a)\right) .
\end{equation*}

\textit{Step 2}. Theorem (\ref{Thm:approx_largeSets_chap2_a_fixed})(i) holds
uniformly in $v$ in $\left( a,a+c\right) $ where $Y_{1}^{k}$ is generated
under $P_{nA}.$ This result follows from a similar argument as used in
Theorem \ref{thm:egal_chap2} where (\ref%
{Thm:approx_largeSets(i)_chap2_a_fixed}) is proved under the local sampling $%
P_{nv}.$ A close look at the proof shows that (\ref%
{Thm:approx_largeSets(i)_chap2_a_fixed}) holds whenever Lemmas 22 and 23, stated in \cite{BroniatowskiCaron2011} for the variables $\mathbf{U}_{i}$'s instead of $\mathbf{X}_{i}$'s hold
under $P_{nA}.$ Those lemmas are substituted by Lemmas \ref%
{LemmaMomentsCond_chap2} and \ref{Lemma_max_U_i_under_E_n_chap2} here above.

Inserting (\ref{Thm:approx_largeSets(i)_chap2_a_fixed}) in (\ref%
{Reductiona+c}) yields%
\begin{align*}
p_{nA}(Y_{1}^{k})& =\left( \int_{a}^{a+c}g_{nv}(Y_{1}^{k})p(\left. \mathbf{U}%
_{1,n}/n=v\right\vert \mathbf{U}_{1,n}>na)dv\right)  \\
& \left( 1+o_{p_{nA}}\left( \max \left( \epsilon _{n}\left( \log n\right)
^{2},\left( \exp \left( -ncm^{-1}(a)\right) \right) ^{\delta }\right) \right)
\right) .
\end{align*}

dor any positive $\delta <1$.

The conditional density of $\mathbf{U}_{1,n}/n$ given $\left( \mathbf{U}%
_{1,n}>na\right) $ is given in (\ref{dens_exp_s}) which holds uniformly in $%
v $ on $(a,a+c).$\newline
\newline

Summing up we have proved 
\begin{equation*}
p_{nA}(Y_{1}^{k})=
\end{equation*}%
\begin{align*}
& \left( nm^{-1}\left( a\right) \int_{a}^{a+c}g_{nv}(Y_{1}^{k})\left( \exp
-nm^{-1}\left( a\right) \left( v-a\right) \right) dv\right)  \\
& \left( 1+o_{p_{nA}}\left( \max \left( \epsilon _{n}\left( \log n\right)
^{2},\left( \exp \left( -ncm^{-1}(a)\right) \right) ^{\delta }\right) \right)
\right) 
\end{align*}%
as $n\rightarrow \infty $ for any positive $\delta .$

In order to get the approximation of $p_{nA}$ by the density $g_{nA}$ it is
enough to observe that 
\begin{equation*}
nm^{-1}\left( a\right) \int_{a}^{a+c}g_{nv}(Y_{1}^{k})\left( \exp
-nm^{-1}\left( a\right) \left( v-a\right) \right) dv
\end{equation*}%
\begin{equation*}
=1+o_{_{P_{nA}}}\left( \exp -ncm^{-1}(a)\right)
\end{equation*}%
as $n\rightarrow \infty $ which completes the proof of (\ref%
{Thm:approx_largeSets(i)_chap2_a_fixed}). The proof of (\ref%
{Thm:approx_largeSets(ii)_chap2_a_fixed}) follows from (\ref%
{Thm:approx_largeSets(i)_chap2_a_fixed}) and Lemma \ref%
{Lemma:commute_from_p_n_to_g_n_chap2} cited hereunder.

The following Lemma proves that approximating $p_{nA}$ by $g_{nA}$ under $%
p_{nA}$ is similar to approximating $p_{nA}$ by $g_{nA}$ under $g_{nA}.$

Let $\mathfrak{R}_{n}$ and $\mathfrak{S}_{n}$ denote two p.m's on $\mathbb{R}%
^{n}$ with respective densities $\mathfrak{r}_{n}$ and $\mathfrak{s}_{n}.$

\begin{lemma}
\label{Lemma:commute_from_p_n_to_g_n_chap2} Suppose that for some sequence $%
\varepsilon_{n}$ which tends to $0$ as $n$ tends to infinity 
\begin{equation*}
\mathfrak{r}_{n}\left( Y_{1}^{n}\right) =\mathfrak{s}_{n}\left(
Y_{1}^{n}\right) \left( 1+o_{\mathfrak{R}_{n}}(\varepsilon_{n})\right)
\end{equation*}
as $n$ tends to $\infty.$ Then 
\begin{equation*}
\mathfrak{s}_{n}\left( Y_{1}^{n}\right) =\mathfrak{r}_{n}\left(
Y_{1}^{n}\right) \left( 1+o_{\mathfrak{S}_{n}}(\varepsilon_{n})\right) .
\end{equation*}
\end{lemma}

\begin{proof}
Denote 
\begin{equation*}
A_{n,\delta\varepsilon_{n}}:=
\end{equation*}
\begin{equation*}
\left\{ y_{1}^{n}:(1-\varepsilon_{n})\mathfrak{s}_{n}\left( y_{1}^{n}\right)
\leq\mathfrak{r}_{n}\left( y_{1}^{n}\right) \leq\mathfrak{s}_{n}\left(
y_{1}^{n}\right) (1+\varepsilon_{n})\right\} . 
\end{equation*}
It holds for all positive $\delta$%
\begin{equation*}
\lim_{n\rightarrow\infty}I(n,\delta)=1 
\end{equation*}
where 
\begin{equation*}
I(n,\delta):=\int\mathds{1}_{A_{n,\delta\varepsilon_{n}}}\left(
y_{1}^{n}\right) \frac{\mathfrak{r}_{n}\left( y_{1}^{n}\right) }{\mathfrak{s}%
_{n}(y_{1}^{n})}\mathfrak{s}_{n}(y_{1}^{n})dy_{1}^{n}. 
\end{equation*}
Since 
\begin{equation*}
I(n,\delta)\leq(1+\delta\varepsilon_{n})\mathfrak{S}_{n}\left( A_{n,\delta
\varepsilon_{n}}\right) 
\end{equation*}
it follows that 
\begin{equation*}
\lim_{n\rightarrow\infty}\mathfrak{S}_{n}\left(
A_{n,\delta\varepsilon_{n}}\right) =1, 
\end{equation*}
which proves the claim.
\end{proof}

\subsection{Proof of Lemma \protect\ref{Lemma set C_n for efficiency}}

Assume $k/n\rightarrow1.$ Let $C_{n}$ in $\mathbb{R}^{n}$ such that for all $%
y_{1}^{n}$ in $C_{n},$%
\begin{equation*}
\left\vert \frac{p_{nA}(y_{1}^{k})}{g_{nA}\left( y_{1}^{k}\right) }%
-1\right\vert <\delta_{n}
\end{equation*}
with $\delta_{n}$ as in (\ref{vitesse_chap2_a_fixed}) and 
\begin{equation*}
\left\vert \frac{m(t^{k})}{a}-1\right\vert <\alpha_{n}
\end{equation*}
where $t^{k}$ is defined through 
\begin{equation*}
m(t^{k}):=\frac{n}{n-k}\left( a-\frac{u_{1,k}}{n}\right)
\end{equation*}
with $u_{1,k}:=\sum_{i=1}^{k}u(y_{i})$ and $\alpha_{n}$ satisfies 
\begin{equation}
\lim_{n\rightarrow\infty}\alpha_{n}=0  \label{alfa1}
\end{equation}
together with%
\begin{equation}
\lim_{n\rightarrow\infty}\alpha_{n}a\sqrt{n-k}=\infty.  \label{alfa2}
\end{equation}
We prove that 
\begin{equation*}
\lim_{n\rightarrow\infty}G_{nA}\left( C_{n}\right) =1.
\end{equation*}
Let 
\begin{equation*}
A_{n,\varepsilon_{n}}:=A_{\varepsilon_{n}}^{k}\times\mathbb{R}^{n-k}
\end{equation*}
with 
\begin{equation*}
A_{\varepsilon_{n}}^{k}:=\left\{ x_{1}^{k}:\left\vert \frac{p_{nA}(x_{1}^{k})%
}{g_{nA}\left( x_{1}^{k}\right) }-1\right\vert <\delta_{n}\right\} .
\end{equation*}
By the above definition 
\begin{equation}
\lim_{n\rightarrow\infty}P_{nA}\left( A_{n,\varepsilon_{n}}\right) =1.
\label{P_n(a_n,e_n)}
\end{equation}
Note also that 
\begin{align*}
G_{nA}\left( A_{n,\varepsilon_{n}}\right) & :=\int\mathds{1}%
_{A_{n,\varepsilon_{n}}}(x_{1}^{n})g_{nA}\left( x_{1}^{n}\right) dx_{1}^{n}
\\
& =\int\mathds{1}_{A_{\varepsilon_{n}}^{k}}(x_{1}^{k})g_{nA}\left(
x_{1}^{k}\right) dx_{1}^{n} \\
& \geq\frac{1}{1+\delta_{n}}\int\mathds{1}_{A_{%
\varepsilon_{n}}^{k}}(x_{1}^{k})p_{nA}(x_{1}^{k})dx_{1}^{k} \\
& =\frac{1}{1+\delta_{n}}\left( 1+o(1)\right)
\end{align*}
which goes to $1$ as $n$ tends to $\infty.$ We have just proved that the
sequence of sets $A_{n,\varepsilon_{n}}$ contains roughly all the sample
paths $X_{1}^{n}$ under the importance sampling density $g_{nA}.$

We use the fact that $t^{k}$ defined through 
\begin{equation*}
m(t^{k})=\frac{n}{n-k}\left( a-\frac{u_{1,k}}{n}\right)
\end{equation*}
is close to $a$ under $p_{nv}$ uniformly upon $v$ in $(a,a+c)$ and integrate
out with respect to the distribution of $\mathbf{U}_{1,n}/n$ conditionally
on $\mathbf{U}_{1,n}/n\in\left( a,a+c\right) .$

Let $\delta_{n}$ tend to $0$ and $\lim_{n\rightarrow\infty}a\alpha_{n}\sqrt{%
n-k}=\infty$ and 
\begin{equation*}
B_{n}:=\left\{ x_{1}^{n}:\left\vert \frac{m(t^{k})}{a}-1\right\vert
<\alpha_{n}\right\} .
\end{equation*}

We prove that on $B_{n}$ 
\begin{equation}
t^{k}s(t^{k})=a\left( 1+o(1)\right)  \label{t_k s_k}
\end{equation}
holds.

By Lemma 22 in \cite{BroniatowskiCaron2011} and integrating w.r.t. $p_{nv}$ on $\left( a,a+c\right) $ it holds, under (\ref{alfa1}) and (\ref{alfa2})%
\begin{equation}
\lim_{n\rightarrow\infty}P_{nA}\left( B_{n}\right) =1.  \label{P_n(B_n)}
\end{equation}

\bigskip There exists $\delta_{n}^{\prime}$ such that for any $x_{1}^{n}$ in 
$B_{n}$ 
\begin{equation}
\left\vert \frac{t^{k}}{a}-1\right\vert <\delta_{n}^{\prime}.
\label{control t_k/a_n}
\end{equation}
Indeed 
\begin{equation*}
\left\vert \frac{m(t^{k})}{a}-1\right\vert =\left\vert \frac{t^{k}\left(
1+v_{k}\right) }{a}-1\right\vert <\alpha_{n}
\end{equation*}
and $\lim_{n\rightarrow\infty}v_{k}=0.$ Therefore%
\begin{equation*}
1-\frac{v_{k}t^{k}}{a}-\alpha_{n}<\frac{t^{k}}{a}<1-\frac{v_{k}t^{k}}{a}%
+\alpha_{n}.
\end{equation*}
Since $\frac{m(t^{k})}{a\text{ }}$ is bounded so is $\frac{t^{k}}{a}$ and
therefore $\frac{v_{k}t^{k}}{a}\rightarrow0$ as $n\rightarrow\infty$ which
implies (\ref{control t_k/a_n}).

Further (\ref{control t_k/a_n}) implies that there exists $\delta_{n}"$ such
that 
\begin{equation*}
\left\vert \frac{t^{k}s(t^{k})}{a}-1\right\vert <\delta_{n}".
\end{equation*}
Indeed 
\begin{align*}
\left\vert \frac{t^{k}s(t^{k})}{a}-1\right\vert & =\left\vert \frac {%
t^{k}\left( 1+u_{k}\right) }{a}-1\right\vert \\
& \leq\delta_{n}^{\prime}+\left( 1+\delta_{n}^{\prime}\right)
u_{k}=\delta_{n}"
\end{align*}
where $\lim_{n\rightarrow\infty}u_{k}=0.$ Therefore (\ref{t_k s_k}) holds.

Define 
\begin{equation*}
C_{n}:=B_{n}\cap A_{n,\varepsilon_{n}}
\end{equation*}
Since 
\begin{equation*}
\int\mathds{1}_{C_{n}}(x_{1}^{n})g_{nA}\left( x_{1}^{k}\right) dx_{1}^{n}\geq%
\frac{1}{1+\delta_{n}}\int\mathds{1}_{C_{n}}p_{nA}(x_{1}^{n})dx_{1}^{n}
\end{equation*}
and by (\ref{P_n(a_n,e_n)}) and (\ref{P_n(B_n)}) 
\begin{equation*}
\lim_{n\rightarrow\infty}P_{nA}\left( C_{n}\right) =1
\end{equation*}
we obtain 
\begin{equation*}
\lim_{n\rightarrow\infty}G_{nA}\left( C_{n}\right) =1.
\end{equation*}
which concludes the proof of (i) and (ii).

\end{document}